\documentclass[10pt,a4paper,twoside]{article}
\usepackage{fullpage}

\usepackage[utf8]{inputenc}
\usepackage{todonotes}
\usepackage{amsmath, amsthm, amssymb}
\usepackage{thmtools, thm-restate}

\usepackage{mathtools}

\usepackage{graphicx}
\usepackage{subcaption}

\usepackage{dsfont}

\usepackage{algorithmic}
\usepackage[norelsize,english,ruled,lined]{algorithm2e}

\usepackage{ifthen}

\usepackage{color}

\usepackage{hyperref}
\hypersetup{
 pdfauthor={O.~Mula},
 pdftitle={},
 pdfsubject={},
 pdfkeywords={}
}

\newcommand{\cB}{\ensuremath{\mathcal{B}}}

\newcommand{\cI}{\ensuremath{\mathcal{I}}}

\newcommand{\cM}{\ensuremath{\mathcal{M}}}

\newcommand{\cN}{\ensuremath{\mathcal{N}}}
\newcommand{\cO}{\ensuremath{\mathcal{O}}}
\newcommand{\cP}{\ensuremath{\mathcal{P}}}

\newcommand{\cT}{\ensuremath{\mathcal{T}}}

\newcommand{\bN}{\ensuremath{\mathbb{N}}}

\newcommand{\bR}{\ensuremath{\mathbb{R}}}

\newcommand{\rT}{\ensuremath{\mathrm{T}}}
\newcommand{\rU}{\ensuremath{\mathrm{U}}}

\newcommand{\rY}{\ensuremath{\mathrm{Y}}}
\newcommand{\rZ}{\ensuremath{\mathrm{Z}}}

\newtheorem{theorem}{Theorem}[section]
\newtheorem{lemma}[theorem]{Lemma}

\newtheorem{corollary}[theorem]{Corollary}

\newtheorem{remark}[theorem]{Remark}

\theoremstyle{definition}
\newtheorem{definition}[theorem]{Definition}

\numberwithin{equation}{section}

\newcommand\cost{\ensuremath{\textrm{cost}}}
\newcommand{\be}{\begin{equation}}
\newcommand{\ee}{\end{equation}}
\newcommand{\Exp}[1]{{\rm Exp}_{#1}}
\newcommand{\Log}[1]{{\rm Log}_{#1}}
\newcommand\cdf[1]{\operatorname{cdf}_{#1}}
\newcommand\icdf[1]{\operatorname{icdf}_{#1}}
\newcommand{\eps}{\ensuremath{\varepsilon}}

\newcommand{\charFun}{\ensuremath{\mathds{1}}}

\newcommand{\opt}{\ensuremath{\mathrm{opt}}}

\DeclareMathOperator*{\argmin}{arg\,min}
\newcommand{\rd}{\ensuremath{\mathrm d}}
\newcommand{\dx}{\ensuremath{\mathrm dx}}

\DeclareMathOperator*{\vspan}{span}

\def\[{\Bigl [}
\def\]{\Bigr ]}
\def\({\Bigl (}
\def\){\Bigr )}
\def\<{\left <}
\def\>{\right >}

\newcommand{\wc}{\ensuremath{\mathrm{wc}}}
\newcommand{\av}{\ensuremath{\mathrm{av}}}
\newcommand{\tr}{\ensuremath{\mathrm{tr}}}
\newcommand{\test}{\ensuremath{\mathrm{test}}}
\newcommand{\gBar}{\ensuremath{\mathrm{gBar}}}
\newcommand{\tPCA}{\ensuremath{\mathrm{tPCA}}}
\newcommand{\PCA}{\ensuremath{\mathrm{PCA}}}
\newcommand{\xmin}{\ensuremath{x_{\min}}}
\newcommand{\xmax}{\ensuremath{x_{\max}}}

\renewcommand{\Pr}{\ensuremath{\cP_2(\Omega)}}
\newcommand{\bary}{\ensuremath{\mathrm{Bar}}}

\newcommand{\Conv}{\ensuremath{\mathrm{Conv}}}

\newcommand{\cMtest}{\ensuremath{{\cM}_{\text{test}}}}

\definecolor{verde}{rgb}{0.2, 0.5, 0.3}

\newcommand\corr[1]{\textcolor{black}{#1}}

\begin{document}
\title{
Nonlinear model reduction on metric spaces.\\Application to one-dimensional conservative PDEs in Wasserstein spaces.
}
\author{V.~Ehrlacher, D.~Lombardi, O.~Mula and F.-X.~Vialard}
\date{}
\maketitle

\abstract{
We consider the problem of model reduction of parametrized PDEs where the goal is to approximate any function belonging to the set of solutions at a reduced computational cost. 
For this, the bottom line of most strategies has so far been based on the approximation of the solution set by linear spaces on Hilbert or Banach spaces. This approach can be expected to be successful only when 
the Kolmogorov width of the set decays fast. While this is the case on certain parabolic or elliptic problems, most transport-dominated problems are expected to present a slow decaying 
width and require to study nonlinear approximation methods. In this work, we propose to address the reduction problem from the perspective of general metric spaces with a suitably defined notion of distance. 
We develop and compare two different approaches, one based on barycenters and another one using tangent spaces when the metric space has an additional Riemannian structure. 
\corr{Since the notion of linear vectorial spaces does not exist in general metric spaces, both approaches result in nonlinear approximation methods}. We give theoretical and numerical evidence of their efficiency to reduce complexity for one-dimensional conservative PDEs where the underlying metric space can be chosen to be the $L^2$-Wasserstein space.
}

\section{Introduction}
In modern applications of science, industry and numerous other fields, the available time for design and decision-making is becoming shorter, and some tasks 
are even required to be performed in real time. The process usually involves predictions of the state of complex systems which, in order to be reliable, need to be described by sophisticated models. 
The predictions are generally the output of inverse or optimal control problems that are formulated on these models and which cannot be solved in real time unless the overall complexity has 
been appropriately reduced. Our focus lies in the case where the model is given by a Partial Differential Equation (PDE) that depends on certain parameters. In this setting, 
the routines for prediction require to evaluate solutions of the PDE on a large set of dynamically updated parameters. This motivates the search for accurate and online methods to approximate the solutions at a reduced computational cost. This task, usually known as \emph{reduced modelling} or \emph{model order reduction}, can be summarized as follows.

Let $\Omega$ be a domain of $\bR^{\corr{D}}$ for a given dimension $\corr{D}\geq 1$ and let $(V,d)$ be a metric space with metric $d$, \corr{containing the set of solutions of a given parametric PDE} defined over the domain $\Omega$. 
The main goal of model reduction is to approximate 
as accurately and quickly as possible the solution $u(z)\in V$ of a problem of the form
\begin{equation}
\label{eq:pde}
\cP(u(z),z) = 0
\end{equation}
for many different values of a vector $z = (z_1,\dots, z_p)$ in a certain range $\rZ \subset \bR^p$. In the above formula, $\cP$ is a differential or integro-differential operator parametrized by $z$, 
and we assume that for each $z \in \rZ$ there exists a unique solution $u(z)\in V$ to problem \eqref{eq:pde}. The set of all solutions is defined as
\begin{equation}
\cM \coloneqq \{ u(z)\: :\: z\in\rZ \} \subset V,
\end{equation}
and is often referred to as the solution manifold with some abuse of terminology\footnote{This set of solutions may not be a differentiable manifold in infinite dimensions.}.

In the context of model reduction, $V$ is traditionally chosen to be a Banach or a Hilbert space with the metric given by its norm denoted by $\Vert \cdot \Vert$. \corr{However, relevant classes of problems could be posed either on Banach spaces or on metric spaces and the latter characterization \emph{may be} more convenient for model reduction in some situations.} To name a few examples involving gradient flows, we cite \cite{GO2001} for Hele-Shaw flows, \cite{GST2009} for quantum problems, \cite{Otto2001} for porous media flows, \cite{JKO1998} for Fokker-Planck equations and \cite{ZM2015, BL2013} for Keller-Segel models in chemotaxis. Other examples involving metric spaces that are not necessarily related to gradient flows are \cite{BF2005} for the Camassa-Holm equation, \cite{CGH2019} for the Hunter-Saxton equation. Such examples can often be interpreted as a geodesic flow on a group of diffeomorphisms and can thus be encoded as Hamiltonian flows. \corr{In addition to this, there are other problems which cannot be defined on Banach vector spaces and can only be defined over metric spaces. Consider for instance the case of a pure transport equation with constant velocity where the initial data is a Dirac measure concentrated on one point. The solution of this PDE remains  at all times a (translated) Dirac mass. More generally, it has been proven that solutions to certain nonlinear dissipative evolution equations with measure-valued initial data are measure-valued and do not belong to some standard Lebesgue or Sobolev spaces. They can however be formulated in the form of Wasserstein gradient flows.}

\corr{It appears therefore that extending the notion of model reduction in Banach/Hilbert spaces to more general metric spaces could enlarge the scope of problems that can potentially be addressed. Since the notion of linear vectorial spaces does not exist in general metric spaces, both approaches result in nonlinear approximation methods.} To develop further on the potential interest brought by nonlinear model reduction on metric spaces, let us briefly recall the classical lines followed for model reduction on Banach/Hilbert spaces. 
Assume for now that $V$ is a Banach space. Most methods are typically based on determining a ``good'' $n$-dimensional subspace $V_n=\vspan\{v_1,\dots, v_n\}\subset V$ that
yields efficient approximations of $u(z)$ in $V_n$ of the form
\begin{equation}
\label{eq:decomposition}
u_n (z) \coloneqq \sum_{i=1}^n c_i(z) v_i
\end{equation}
for some coefficients $c_1(z),\cdots, c_n(z)\in \mathbb{R }$. This approach is the backbone of most existing methods among which stand the reduced basis method (\cite{HRS2015,QMN2016}), the empirical interpolation method and 
its generalized version (G-EIM, \cite{BMNP2004,GMNP2007,MM2013,MMT2016}), Principal Component Analysis (PCA, see \cite[Chapter 1]{BCOW2017}), polynomial-based methods like \cite{CDS2010,CDS2011} or low-rank methods (\cite{KS2011}).

The approximation quality of the obtained subspace $V_n$ is either measured through the worst case error
\begin{equation}
\label{eq:wc}
e_{\wc}(\cM, V, V_n)
\coloneqq
\sup_{z\in\rZ} \inf_{w_n\in V_n} d\( u(z), w_n \),
\end{equation}
or the average error
\begin{equation}
\label{eq:av}
e_{\av}(\cM, V, V_n)
\coloneqq
\left(
\int_{z\in\rZ} \inf_{w_n\in V_n}  d^2\( u(z), w_n \)\,\mathrm{d} \mu(z)
\right)^{1/2},
\end{equation}
where $\mu$ is a probability measure on $\rZ$, given a priori and from which the parameters are sampled.

The reduction method is considered efficient if $e_{\wc}(\cM, V, V_n)$ (or $e_{\av}(\cM, V, V_n)$) decays rapidly to 0 as $n$ goes to $\infty$. There is sound evidence of efficiency only in the case of certain elliptic and parabolic PDEs. More precisely, it has been shown in~\cite{CD2015} that for this type of equations, under suitable assumptions, the $L^\infty$ Kolmogorov width defined as 
\begin{equation}\label{eq:widthLinfty}
d_n(\cM, V) \coloneqq \inf_{\begin{array}{c}
       V_n \subset V,\\
       \mbox{\rm dim}\; V_n = n\\
      \end{array}} e_{\wc}(\cM, V, V_n)
\end{equation}
and the $L^2$ Kolmogorov width
\begin{equation}
\label{eq:widthL2}
\delta_n(\cM, V) \coloneqq \inf_{\begin{array}{c}
       V_n \subset V,\\
       \mbox{\rm dim}\; V_n = n\\
      \end{array}} e_{\av}(\cM, V, V_n)
\end{equation}
decay exponentially or polynomially with high exponent as $n$ grows. In the context of model reduction, this quantity gives the best possible performance that one can achieve when approximating $\cM$ with $n$-dimensional linear spaces.

Optimal linear subspaces $V_n\subset V$ of dimension $n$ which realize the infimum of \eqref{eq:widthLinfty} cannot be computed in practice in general. However, it has been shown that greedy algorithms can be used to build sequences of linear spaces $(V_n)_{n\geq 1}$ whose approximation error $e_{\wc}(\cM, V_n)$ decay at a comparable rate as the Kolmogorov $n$-width $d_n(\cM,V)$. These algorithms are the backbone of the so-called Reduced Basis method~\cite{BCDDPW2011}. In the case of \eqref{eq:widthL2}, the optimal subspaces for which the minimum is attained are obtained using
the PCA \corr{or Proper Orthogonal Decomposition (POD)} method.

In this paper, our goal is to extend the above notion of model reduction to more general metric spaces in view of the following facts. 
First of all, in the context of Banach or Hilbert spaces, linear methods are unfortunately not well suited for hyperbolic problems. Among others, this is due to the transport of shock discontinuities 
whose locations may vary together with the parameters. \corr{It was proved in~\cite[Chapter 3, see equation (3.76)]{BCOW2017} that }the $L^\infty$ Kolmogorov width of simple pure transport problems decays very slowly, at a rate $n^{-1/2}$ if $V=L^2$ (similar examples can be found in \cite{Welper2015, OR2016, BCMN2018}). The same type of result has recently been 
derived for wave propagation problems in \cite{GK2019}. These results highlight that linear methods of the type \eqref{eq:decomposition} are not expected to provide a fast decay in numerous 
transport dominated problems, and may be highly suboptimal in terms of the trade off between accuracy and numerical complexity. For these classes of problems, an efficient strategy 
for model reduction requires to look for \emph{nonlinear methods} that capture the geometry of $\cM$ in a finer manner than linear spaces. In addition to 
the idea of searching for nonlinear methods, it may be beneficial to move from the classical Banach/Hilbert metric framework to more general metric spaces in order to better quantify the ability to capture 
specific important features like translations or shifts. Finally, as already brought up, this broader setting enlarges the scope of problems that can be treated.

We next describe the state of the art methods proposed to go beyond the linear case and to address hyperbolic problems. Then, we summarize the contributions and the organization of this paper.

\subsection{State of the art}
\indent
\par{\textbf{Works on adaptivity.}}
Several works try to circumvent the poor representation given by linear spaces by building local spaces. In \cite{carlberg2015adaptive} a strategy inspired by the mesh $h-$refinement is proposed. In \cite{amsallem2012nonlinear,amsallem2016pebl} a construction of linear subspaces is proposed, based on a $k$-means clustering that groups together similar snapshots of the solution. A similar approach is presented in \cite{peherstorfer2014localized} to overcome some shortcomings of the DEIM approach. In the two recent preprints \cite{lee2018model,gonzalez2018learning}, the subspaces to be used for model reduction are identified by using autoencoders and neural networks. 

\par{\textbf{Stabilization strategies.}}
Advection dominated problems tend to show potential instabilities when reduced-order integrators are built. In \cite{maday2016online, torlo2018stabilized} an online stabilization is proposed. A form of stabilization based on an $L^1$ minimisation problem is proposed in \cite{abgrall2016robust}.

\par{\textbf{Reduced-order modeling from a geometrical point of view.}}
Several works in the literature propose to transform the snapshots by making a Lie group acting on them. In \cite{mowlavi2018model} the authors propose a method to get rid of continuous symmetries in parametric dynamical systems whose dynamics is invariant under continuous group action\footnote{On the KdV equations these symmetries were studied analytically in \cite{gazeau1992symmetries}}. A similar approach was proposed in \cite{ohlberger2013nonlinear}.

Conversely to transporting the snapshots, dynamical bases approaches are defined in which the subspace used to approximate the solution evolves in time. An example of such a method is the Dynamically Orthogonal (DO) decomposition (detailed in \cite{koch2007,KL2010,MNZ2015,FL2018}). We also cite \cite{AF2008} and a recent extension \cite{MHHA2018}, which have focused on model reduction over the Grassman manifold.

In the context of Hamiltonian systems, \cite{afkham2017structure, HP2018} introduce a reduced-order framework that preserves the symplectic structure of the dynamical system.

\par{\textbf{Non-linear transformations.}}
In \cite{Welper2015,Welper2017}, a transformed snapshot interpolation method is introduced, aiming at finding a set of non-linear transformations such that the approximation of the transformed snapshots by a linear combination of modes is efficient. In \cite{nair2018transported}, which addresses compressible fluid-dynamics, the set of the snapshots is transported into a reference configuration by a displacement field which is identified through a polynomial expansion. In \cite{IL2014}, the authors propose a non-linear model reduction strategy based on optimal transport and the reduction of optimal transport maps. Another strategy based on non-linear transformations was proposed in \cite{cagniart2019model} to deal with hyperbolic problems with applications to transonic flows around airfoils. Finally, the recent work  \cite{cagniart2019model} introduces a general framework to look for reference configurations and for transformations depending upon few parameters in order to deal with advection dominated problems.
  

\subsection{Contribution and organization of the paper}
\textbf{The main contribution of this work is to develop the idea of reduced modeling in metric spaces.} For this:
\begin{itemize}
\item We give theoretical evidence on two simple PDEs (pure transport and Burgers' equation in 1D) that it is beneficial to do model reduction on metric spaces rather than on more classical Banach spaces.
\item We develop two model reduction strategies on general metric spaces, one based on tangent spaces, the second one based on barycenters.
\end{itemize}
Our approach is, to the best of our knowledge, novel in the field of reduced modelling. The first approach is however related to the so-called tangent PCA, which has drawn significant 
interest in numerous fields like pattern recognition, shape analysis, medical imaging, computer vision \cite{Fletcher2004PGA,Sommer2014}. More recently, it has also been used 
in statistics \cite{HuckemannPCA} and machine learning to study families of histograms and probability densities \cite{CSBCP2017}. Our second approach based on barycenters is entirely
novel to the best of our knowledge, and it could be used as an alternative to tPCA in other applications apart from model reduction.

As a support for our numerical tests, we will consider the model reduction of one-dimensional conservative PDEs in the $L^2$-Wasserstein metric. Time will be seen as one of 
the parameters. The precise setting and notation is the following. For a given open interval $\Omega\subset \bR$ and a given final time $T>0$, \corr{we define $\mathcal{D}'(\Omega)$ the set of real-valued distributions defined on $\Omega$ 
and consider a 1d conservative PDE that depends on $\tilde p\in \bN^*\corr{:=\bN\setminus \{0\}}$ parameters. These parameters belong to a compact set denoted by $\rY\subset \mathbb{R}^{\tilde p}$. For any $y\in \rY$, 
$u_y: [0,T] \to \mathcal{D}'(\Omega)$ is the (distribution-valued) solution to
\begin{equation}\label{eq:oned}
\partial_t u_y(t) - \partial_x F(u_y(t);y,t) =0 \quad \mbox{ in } \mathcal D'(\Omega), \; \forall t\in [0,T],
\end{equation}}
with appropriate initial and boundary conditions. We assume that $F(\cdot;y,t)$ is a real-valued mapping defined on a set of \corr{distributions} defined on $\Omega$ so that the solution to \eqref{eq:oned} is well-defined.

\corr{
We denote by $\Pr$ the set of probability measures on $\Omega$ with finite second-order moments. With a slight abuse of notation, $\Pr$ can be seen as a subset of 
$\mathcal{D}'(\Omega)$ in the sense that for all $\nu \in \Pr$, the linear application $T_\nu: \mathcal{D}(\Omega) \to \bR$ defined by
$$
\forall \phi \in \mathcal{D}(\Omega), \quad T_\nu(\phi) = \int_\Omega \phi \,d\nu
$$
defines a distribution on $\Omega$.
}

\corr{
In the following, we assume that
\begin{equation}\label{eq:mainass}
\forall (t,y)\in Z\coloneqq [0,T]\times Y, \quad u_y(t) = T_{\nu_y(t)} \quad \mbox{ for some }\nu_y(t)\in \Pr .
\end{equation}
In order not to overload notation, we will write the measure $\nu_y(t)\in \Pr$ with the notation $u_y(t)$ with a slight abuse of notation. Thus, $u_y(t)$ will denote the measure $\nu_y(t)\in \Pr$ such that \eqref{eq:mainass} holds. In the case when an element $u\in \Pr$ is absolutely continuous with respect to the Lebesgue measure, we denote by $\rho_u$ its density, so that $\,du(x) = \rho_u(x)\,dx$.
}

Since time is an additional parameter, our parameter set is
$$
Z \coloneqq [0,T] \times \rY \subset \bR^p,\quad p\coloneqq \tilde p +1
$$
thus the solution set is 
$$
\cM \coloneqq \{ u(z) = u_y(t) \in \Pr \,:\, z=(t,y) \in Z \} \subset \Pr.
$$
\textbf{The paper is organized as follows:}
We begin by recalling in Section~\ref{sec:nota} basic notions on metric spaces, and more specifically on the $L^2$-Wasserstein space in one dimension.
We then highlight in Section~\ref{sec:kolmo} the interest of doing model reduction of conservative PDEs in metric spaces. For two simple PDEs (a pure transport equation and an inviscid Burgers' equation in 1D), 
we prove theoretically that this point of view yields better approximation properties than classical linear methods on Banach spaces. In Section \ref{sec:algo}, we describe 
the two numerical methods (tPCA and gBar) that we propose for model reduction in general metric spaces. We also make the different steps of the algorithms explicit in the particular 
case of the $L^2$ Wasserstein space in 1D. \corr{One distinguishing feature of the methods is that they are purely \emph{data-driven} in the sense that the online phase does not involve solving the original PDE in any reduced space/manifold. The methods could thus be seen as a fast \emph{emulator} of the high-fidelity solutions. This is in contrast to traditional projection-based approaches which require solving the original equations in the online phase.} 
Finally, Section \ref{sec:num} gives numerical evidence of the performance of the two algorithms on four different problems: an inviscid and viscous Burgers' equation, a Camassa-Holm 
equation and a Korteveg-de-Vries (KdV) equation. The results illustrate the ability of the proposed methods to capture transport phenomena \corr{at a very reduced computational cost. For the example of the viscous Burgers' equations, the average online computational time of one snapshot is reduced by a factor of about 100 with respect to the high-fidelity computation.} The KdV problem shows some limitations of the methodology in terms of capturing the fusion and separation of peakons and motivates our final section \ref{sec:extensions}, where we list possible extensions to address not only this issue, but also to treat more general multi-dimensional and non-conservative problems.

%
%
%

\section{Metric spaces and Wasserstein space}\label{sec:nota}

Since in our numerical examples we focus on the model reduction of parametric one-dimensional conservative PDEs in the $L^2$-Wasserstein space, 
we start this section by recalling some basic properties of this space. We next recall the definition of exponential and logarithmic maps on general Riemannian manifolds, 
and then detail their expression in the case of the 1D Wasserstein space. Finally, we introduce the notion of barycenters on general metric spaces.


\subsection{Definition of the $L^2$-Wasserstein space in one dimension}\label{sec:not}
Let $\Omega = [\xmin, \xmax] \subset \bR\text{, with } -\infty \leq \xmin<\xmax \leq \infty$ be the domain of interest.
Let $\Pr$ denote the set of probability measures on $\Omega$ with finite second-order moments. For all $u\in \Pr$, we denote by
\begin{equation}
\operatorname{cdf}_u : \left\{
\begin{array}{ccc}
 \Omega &\to & [0,1]  \\
 x & \mapsto & \operatorname{cdf}_u(x) \coloneqq \int_{\xmin}^x\corr{ \,du}\\
\end{array}\right.
\end{equation}
its cumulative distribution function (cdf), and by 
\begin{equation}
\operatorname{icdf}_u: 
\left\{
\begin{array}{ccc}
[0,1] &\to & \Omega  \\
s & \mapsto & \operatorname{cdf}_u^{-1}(s)\coloneqq\inf\{ x\in \Omega, \; \operatorname{cdf}_u(x)>s\}\\
\end{array} \right.
\end{equation}
the generalized inverse of the cdf (icdf). The $L^2$-Wasserstein distance is defined by
$$
W_2(u,v) \coloneqq \mathop{\inf}_{\pi \in \Pi(u,v)} \left( \int_{\Omega \times \Omega} (x-y)^2 \,d\pi(x,y) \right)^{1/2},\quad \forall (u,v)\in \Pr \times \Pr,
$$
where $\Pi(u,v)$ is the set of probability measures on $\Omega \times \Omega$ with marginals $u$ and $v$. In the particular case of one dimensional marginal domains, it can be equivalently expressed using the inverse cumulative distribution functions  as
\begin{equation}\label{EqFlatReduction}
W_2(u,v) = \|\operatorname{icdf}_u - \operatorname{icdf}_v\|_{L^2([0,1])}.
\end{equation}
The space $\Pr$ endowed with the distance $W_2$ is a metric space, usually called $L^2$-Wasserstein space (see \cite{Villani2003} for more details). 

\subsection{Exponential and logarithmic maps}\label{sec:exp}

Let $(V,d)$ be a metric space where $V$ is a Riemannian manifold. 
For any $w\in V$, we denote by $T_w V$ the tangent space to $V$ at point $w$. Then, there exists a subset $\mathcal{O}_w \subset T_wV$ such that the exponential map 
$\Exp{w}: \mathcal{O}_w  \to V$ can be defined and is a diffeomorphism onto its image. 
The logarithmic map $\Log{w}: \Exp{w}(\mathcal{O}_w) \to \mathcal{O}_w$ is then defined as the generalized inverse of $\Exp{w}$. 

\medskip
In the case when $(V,d) = (\Pr,W_2)$, the exponential and logarithmic maps have a particular simple form which we explain next.
We can take advantage of the fact that, after composition with the nonlinear map $\Pr \ni u \mapsto \operatorname{icdf}_u \in L^2([0,1])$, the space $(\Pr,W_2)$ is isometric to $(\cI, \Vert \cdot \Vert_{L^2([0,1])})$, where
\begin{equation}
\mathcal{I} \coloneqq \{ \operatorname{icdf}_u, \; u \in \Pr\}
\end{equation}
is a convex subset of $L^2([0,1])$. The following well-known result then holds.
\begin{theorem}\label{ThIsometry}
The map $\operatorname{icdf}: \Pr \to \mathcal{I}$ defined by $\operatorname{icdf}(u) = \operatorname{icdf}_u$ is an (homeomorphism) isometry between $(\Pr,W_2)$ and $(\mathcal{I},\| \cdot \|_{L^2([0,1])})$. 
\end{theorem}
The fact that the map $\operatorname{icdf}$ is an isometry is actually a consequence of \eqref{EqFlatReduction}. Note that the inverse of $\operatorname{icdf}$ can be defined\footnote{The generalized inverse of $f$ is a monotonic function which is of bounded variation and thus has a well-defined non-negative measure as derivative.} as
$$
\operatorname{icdf}^{-1}(f) = \frac{\rd}{\dx} [f^{-1}], \quad \forall f\in \mathcal{I}.
$$
For all $f_0\in \mathcal{I}$, we introduce 
\begin{equation}
\mathcal{K}_{f_0} = \{f \in L^2([0,1];\Omega) \, :\, f_0+f \in \mathcal{I} \}. 
\end{equation}
which is a closed convex subset of $T_{f_0}\mathcal{I}$.

For a given element $w\in \Pr$, using a slight abuse of notation, the exponential map can be advantageously defined on $\mathcal{K}_{\operatorname{icdf}_{w}}$. Indeed, the definition domain $\mathcal{O}_w\subset T_w \mathcal{P}_2(\Omega)$ of 
$\Exp{w}$ is isomorphic to $\mathcal{K}_{\operatorname{icdf}_{w}}$. Using another slight abuse of notation, we will denote in the sequel $\mathcal{O}_w\coloneqq \mathcal{K}_{\operatorname{icdf}_{w}}$. This leads us to the following definition of the exponential and logarithmic maps.

\begin{definition}
Let $w \in \Pr$ be a probability measure. The exponential map $\Exp{w}: \mathcal{O}_w \to \Pr$ is defined as
\begin{equation}
\label{eq:exp}
\Exp{w}(f) = \operatorname{icdf}^{-1}(\operatorname{icdf}_w+ f),\quad \forall f \in \mathcal{O}_{w}.
\end{equation}
It is surjective and we can define its inverse $\Log{w}: \Pr \to \mathcal{O}_w$ as
\begin{equation}
\label{eq:log}
\Log{w}(u) = \operatorname{icdf}_u - \operatorname{icdf}_w,\quad \forall u \in \Pr \,.
\end{equation}
\end{definition}
As a consequence of Theorem \ref{ThIsometry}, the following result holds.
\begin{corollary}\label{ThExp}
The exponential map $\Exp{w}$ is an isometric homeomorphism between $\mathcal{O}_{w}$ and $\Pr$ with inverse $\Log{w}$ and it holds that 
$$
W_2(u,v)=\left\| \Log{w}(u)-\Log{w}(v)\right\|_{L^2([0,1])}, \quad \forall(u,v) \in\Pr\times\Pr.
$$
\end{corollary}
For the proofs of Theorems \ref{ThExp} and \ref{ThIsometry}, we refer to \cite{BGKL2017}. We next give two simple common examples of logarithmic and exponential maps:
\begin{itemize}
\item \textbf{Dirac masses:} Let $\Omega=\bR$ and consider the family of Dirac masses $\{\delta_x \,:\, x\in \bR\} \subset \Pr$. For any $x\in \bR$ and $s\in (0,1)$, $\icdf{\delta_x}(s) = x$,
thus $W_2(\delta_{x_1}, \delta_{x_2}) = |x_1-x_2|$ for any $(x_1,x_2)\in \bR\times\bR$. From \eqref{eq:log}, for all $s\in (0,1)$, $\Log{\delta_{x_2}} (\delta_{x_1})(s) = x_1 -x_2$.
\item \textbf{Translations and rescaling:} 
Let $w\in \cP_2(\bR)$ with density \corr{$\rho_w$} and, for $(a, b) \in (0,\infty)\times \bR$, let $w^{(a,b)}$ be the probability measure with density given by
$$
\corr{\rho}^{(a, b)}(x) = \frac 1 a \rho_w \((x-b)/a\),\quad 
$$
In other words, $\{w^{(a,b)}\,:\, (a, b) \in (0,\infty)\times \bR\}$ is the family of shifted and rescaled probabilities around $w$.
Then, 
\begin{equation}\label{eq:simple}
\operatorname{cdf}_{w^{(a,b)}}(x) \coloneqq  \cdf{w}\left(ax+b\right),\; \forall x \in \bR, \quad \operatorname{icdf}_{w^{(a,b)}}(s) \coloneqq \frac{\icdf{w}\left(s\right) -b}{a}, \; \forall s\in [0,1].
\end{equation}
\corr{
Henceforth, the logarithmic map can be written as:
\begin{equation}
\mathrm{Log}_{w}(w^{(a,b)}) = (a^{-1}-1)\icdf{w}(s) - b/a.
\end{equation}
}
\end{itemize}

\subsection{Barycenters}

We next introduce the notion of barycenters in a general metric space $(V,d)$. Let $n\in \bN^*$ and let
$$
\Sigma_n\coloneqq\left\{ (\lambda_1,\cdots,\lambda_n) \in [0,1]^n, \quad \sum_{i=1}^n \lambda_i = 1\right\}
$$
be the set of barycentric weights. For any $\rU_n = (u_i)_{1\leq i\leq n} \in V^n$ and barycentric weights 
$\Lambda_n = (\lambda_i)_{1\leq i\leq n} \in \Sigma_n$, an associated barycenter is an element of $V$ which minimizes 
\begin{equation}
\label{eq:barygen}
\inf_{v\in V} \sum_{i=1}^n \lambda_i d(v,u_i)^2.
\end{equation}
In full generality, minimizers to \eqref{eq:barygen} may not be unique. In the following, we call $\bary(\rU_n, \Lambda_n)$ the set of minimizers to \eqref{eq:barygen}, which is the set of barycenters of $\rU_n$ with barycentric weights $\Lambda_n$. 

It will be useful to introduce the notion of \emph{optimal barycenter} of an element $u\in V$ for a given family $\rU_n \in V^n$. The set of barycenters with respect to $\rU_n$ is
$$
\cB_n(\rU_n) \coloneqq \bigcup_{\Lambda_n\in \Sigma_n} \bary(\rU_n, \Lambda_n)
$$
and an optimal barycenter of the function $u\in V$ with respect to the set $\rU_n$ is a minimizer of
\begin{equation}\label{eq:optbarygen}
\min_{b \in \cB_n(\rU_n)} d(u, b)^2.
\end{equation}
In other words, a minimizer to \eqref{eq:optbarygen} is the projection of $u$ on the set of barycenters $\cB_n(\rU_n)$.

\medskip

We next present some properties of barycenters in the Wasserstein space $(V,d) = (\Pr,W_2)$ which will be relevant for our developments (see \cite{AC2011} for further details). The first property is that problem \eqref{eq:barygen} has a unique solution, that is, for a family of probability measures $\rU_n = (u_i)_{1\leq i\leq n} \in \Pr^n$ and barycentric weights 
$\Lambda_n = (\lambda_i)_{1\leq i\leq n} \in \Sigma_n$, there exists a unique minimizer to
\begin{equation}
\label{eq:bary}
\min_{v\in \Pr} \sum_{i=1}^n \lambda_i W_2(v,u_i)^2,
\end{equation}
which is denoted by $\bary(\rU_n, \Lambda_n)$.
In addition, the barycenter can be easily characterized in terms of its inverse cumulative distribution function since \eqref{eq:bary} implies that
\begin{equation}
\label{eq:bary-icdf}
\icdf{\bary(\rU_n, \Lambda_n)} = \argmin_{f\in L^2([0,1])} \sum_{i=1}^n \lambda_i \|\icdf{u_i} -f\|_{L^2([0,1])}^2,
\end{equation}
which yields
\begin{equation}
\label{eq:bary-icdf-2}
\icdf{\bary(\rU_n, \Lambda_n)} =\sum_{i=1}^n \lambda_i \icdf{u_i}.
\end{equation}

The \emph{optimal barycenter} of a function $u\in\Pr$ for a given set of functions $\rU_n$ is unique. We denote it $b(u,\rU_n)$ and it can be easily characterized in terms of its inverse cumulative distribution function. Indeed, the minimization problem \eqref{eq:optbarygen}reads in this case
$$
\min_{b \in \cB_n(\rU_n)} W_2(u, b)^2
$$
and it has a unique minimizer $b(u, \rU_n)$. An alternative formulation of the optimal barycenter  is by finding first the optimal weights 
\begin{equation}\label{eq:minbary}
\Lambda_n^{\rm opt}  = \argmin_{\Lambda_n \in \Sigma_n}  W^2_2(u, \bary(\rU_n, \Lambda_n)).
\end{equation}
The optimal barycenter is then
$$
b(u,\rU_n) = \bary(\rU_n, \Lambda^{\opt}_n).
$$
Note that for all $\Lambda_n\in \Sigma_n$ and all $w\in \Pr$,
\begin{align*}
 W^2_2(u, \bary(\rU_n, \Lambda_n))& = \left\| \icdf{u} - \sum_{i=1}^n \lambda_i \icdf{u_i}\right\|^2_{L^2([0,1])}\\
 & =  \left\| \Log{w}(u) - \sum_{i=1}^n \lambda_i \Log{w}(u_i)\right\|^2_{L^2([0,1])}
\end{align*}
so the computation of the optimal weights $\Lambda_n^{\opt}$ in problem \eqref{eq:minbary} is a simple convex quadratic optimization problem.

For a given $w\in\Pr$, denoting
$$
\rT_n = \Log{w}(\rU_n) = \{ \Log{w}(u_1),\dots, \Log{w}(u_n) \}
$$
the logarithmic image of $\rU_n$ and $\Conv(\rT_n)$
the convex hull of $\rT_n$, we see that $\Log{w}(b(u,\rU_n))$ is the projection of $\Log{w}(u)$ onto $\Conv(\rT_n)$, namely
\begin{equation}
\label{eq:bary-proj-hull}
\Log{w}(b(u,\rU_n)) = \argmin_{f\in \Conv(\rT_n)} \left\| \Log{w}(u) - f \right\|^2_{L^2([0,1])}.
\end{equation}

\section{Kolmogorov $n$-widths for two simple conservative PDEs}
\label{sec:kolmo}

In the sequel, $\mathcal{M}$ is the set of solutions of a parametric conservative PDE in one dimension. Instead of working in the usual Banach/Hilbert setting, we assume that $\mathcal{M} \subset \mathcal P_2(\Omega)$. We denote
$$
\mathcal{T}\coloneqq \Log{w}(\mathcal M) \subset L^2([0,1]),
$$
the image of $\mathcal M$ by the logarithmic map $\Log{w}$, where $w$ is an element of $\mathcal P_2(\Omega)$ which will be fixed later on.

To illustrate the interest of working with this metric, we show in a pure transport equation and in an inviscid Burgers' equation that the Kolmogorov widths $d_n(\cT, L^2([0,1]))$ and $\delta_n(\cT, L^2([0,1]))$ decay at a faster rate than the widths $d_n(\cM, L^2(\Omega))$ and $\delta_n(\cM, L^2(\Omega)$ of the original set of solutions $\cM$. This shows that it is convenient to transform the orginal data $\cM$ to $\cT$ by the nonlinear logarithmic mapping before performing the dimensionality reduction. Indeed, if $d_n(\cT, L^2([0,1]))$ (or $\delta_n(\cT, L^2([0,1]))$) decay fast, then there exist spaces $V_n\subset L^2([0,1])$ such that
$$
e_{\wc}(\cT, L^2([0,1]), V_n)=\sup_{f\in\cT}\Vert f - P_{V_n}f \Vert_{L^2([0,1])}
$$
decays fast as $n\to\infty$. Thus if for all $f\in\cT$ the projections $P_{V_n}f \in \cO_w$, then their exponential map is well defined and we can approximate any $u\in\cM$ with $\Exp{w}(P_{V_n}\Log{w}(u))$. Due to the isometric properties of $\Exp{w}$ (see corollary \ref{ThExp}), the approximation error of $\cM$ is
$$
e_{\wc}(\cM, W_2, V_n)
\coloneqq
\sup_{u\in\cM} W_2(u, \Exp{w}(P_{V_n}\Log{w}(u))) = e_{\wc}(\cT, L^2([0,1]), V_n).
$$
Thus the existence of linear spaces $(V_n)_{n\geq 1}$ in $L^2([0,1])$ with good approximation properties for $\cT$ automatically yields good low rank nonlinear approximations for $\cM$ (provided that the exponential map is well defined).

\corr{
The main reason why one can expect that the Kolmogorov widths decay faster after a logarithmic transformation in transport-dominated problems is connected to the fact that sets of \emph{translated} functions, say for instance a set of the form 
$\{ u(\cdot -\tau), \quad \tau \in [0,1]\}$ for some $u\in \mathcal P_2(\Omega)$, are not well approximated by low-dimensional linear spaces. On the contrary, the decay of the widths of the same type of set after logarithmic transformation can become dramatically faster as we illustrate in the pure transport problem of Section \ref{sec:neednonlinear}.
}

\subsection{A pure transport problem}
\label{sec:neednonlinear}

We consider here a prototypical example similar to the one given in \cite{BCMN2018} (see also \cite{Welper2015,GK2019} for other examples). 
\corr{Consider the univariate transport equation
$$
\partial_t \rho_y(t,x)+y \partial_x \rho_y(t,x)=0, \quad x\in\bR, \; t\geq 0,
$$
with initial value $\rho_0(x)=\charFun_{]-1,0]}$, and parameter $y\in \rY\coloneqq[0,1]$. Note that this is a conservative problem since for all $t\geq 0$, $\int_\bR u(t,x)\dx=1$. We consider the parametrized family of solutions at time $t=1$ restricted to $x\in \Omega\coloneqq [-1,1]$, that is
$$
\cM = \{ u(z) = \rho_y(t=1, x) dx \coloneqq\charFun_{[y-1,y]}(x) dx\; : \; z=(t,y)\in \{1\}\times [0,1]\}.
$$
}
Note that here $t$ is fixed so it is not a varying parameter. We have kept it in the notation in order to remain consistent with the notation of the introduction.

Since $\cM\subset\cP_2([-1,1])$, we can define
$$
\cT = \left\{\Log{w}(u(z)) ,\, z \in \rZ\right\} = \left\{ \icdf{u(z)} - \icdf{w},\; y \in[0,1] \right\},
$$
where $w$ is chosen as $\charFun_{[y_0-1,y_0]}$ for some $y_0\in [0,1]$. 

In the following theorem, we state two extreme convergence results. 
On the one hand we prove that $d_n(\cT, L^2([0,1]))=0$ for $n>1$. On the other hand, it holds that $d_n(\cM, L^2(\Omega))$ cannot decay faster that the rate $n^{-1/2}$. We state the result in the $L^2$ metric since it is very commonly considered but a similar reasoning would give a rate of $n^{-1}$ in $L^1$ (see, e.g., \cite{Welper2015}). 
This rigorously proves that standard linear methods cannot be competitive for reducing this type of problems and that shifting the problem from $\cM$ to $\cT$ dramatically helps in reducing the complexity in the pure transport case.

\begin{theorem}
 There exists a constant $c>0$ such that for all $n\in \mathbb{N}^*$,
\begin{equation}
\label{eq:kolmoL2trans}
 d_n(\cM,L^2(\Omega)) \geq \delta_n(\cM,L^2(\Omega)) \geq c n^{-1/2}.
\end{equation}
In addition, 
\begin{equation}\label{eq:deux}
 \forall n>1,\quad d_n(\cT, L^2([0,1]))=0.
\end{equation}
\end{theorem}

\begin{proof}
\corr{Bound \eqref{eq:kolmoL2trans} was first proved in \cite{OR2016}. We provide an alternative proof in Appendix \ref{appendix:proof}.}
\corr{
Let us prove that $d_n(\cT, L^2([0,1]))=0$ for $n>1$. Since for every $y \in [0,1]$, $\charFun_{[y-1,y]} = \charFun_{[y_0-1,y_0]}(x - y + y_0)$, using \eqref{eq:simple}, it holds that
$$
\operatorname{icdf}_{u_z}(s)  = \operatorname{icdf}_{w} - y_0 + y,\quad \forall s \in [0,1],
$$
and for all $s\in [0,1]$,
$$
\Log{w} (u_z)(s)
=
\operatorname{icdf}_{u_z}^{-1}(s) 
-
\operatorname{icdf}_{w}^{-1}(s)
= y - y_0.
$$
}
As a consequence, the set $\cT$ is contained in the one-dimensional space of constant functions defined on $(0,1)$ and $d_n(\cT, L^2([0,1]))=0$ for all $n > 1$.
\end{proof}

\subsection{An inviscid Burgers' equation}\label{sec:inviscid}
In this section, we consider a simple inviscid Burgers' equation. We study a simple though representative example where \corr{we prove a priori estimates on the decay of the Kolmogorov $n$-width 
of $\mathcal{T}$ as $n$ goes to infinity. We are not able to prove that the Kolmogorov $n$-width of the set $\mathcal M$ decays slower than the one of the set $\mathcal T$, but we present in Section~\ref{sec:burgers-num} numerical tests on this particular example
which indicate that this is indeed the case (see Figure~\ref{fig:decay-svd-Burgers}).}

\medskip

Let $\rY=[1/2,3]$, and for all $y\in \rY$, we consider the inviscid Burgers' equation for $(t,x)\in[0,T]\times \Omega=[0,5]\times [-1,4]$,
\begin{align}
\label{eq:burgers}
\partial_t \rho_y + \frac 1 2 \partial_x (\rho_y^2) =0,
\quad
\rho_y(t=0,x) =
\begin{cases}
&0,\quad -1\leq x <0 \\
&y,\quad 0\leq x < \frac 1 y \\
&0,\quad  \frac 1 y \leq x \leq 4,
\end{cases}
\end{align}
with periodic boundary conditions on $\Omega$. \corr{For every $t\in[0,T]$, the solution $\rho_y(t)$ is the density of an absolutely continuous measure $u_y(t) \in \Pr$.}

Problem \eqref{eq:burgers} has a unique entropic solution which reads as follows. For $0<t<2/y$, a wave composed of a shock front and a rarefaction wave propagates from left to right and
\begin{align}
\rho_y(t,x) =
\begin{cases}
0,&\quad -1\leq x <0 \\
\frac x t,&\quad 0\leq x<yt \\
y,&\quad yt\leq x\leq \frac 1 y +\frac{yt}{2},\qquad 0<t<2/y^2. \\
0,&\quad \frac 1 y +\frac{yt}{2} < x \leq 4
\end{cases}
\end{align}
The rarefaction wave reaches the front at $t=2/y^2$ so that for $t\geq 2/y^2$,
\begin{align}
\rho_y(t,x) =
\begin{cases}
0,&\quad -1\leq x <0 \\
\frac x t,&\quad 0\leq x\leq \sqrt{2t},\qquad \forall t\geq 2/y^2 \\
0,&\quad x>\sqrt{2t}
\end{cases}
\end{align}
Let us denote $u_{y,t}(z) := \rho_y(t,\cdot)dx$ the measure associated with $\rho_y$.
The cumulative distribution function $\cdf{u_{y,t}}:\Omega\to [0,1]$ of $u_{y,t}$ is equal to
\begin{align}
\cdf{u_{y,t}} (x) =
\begin{cases}
0,&\quad -1\leq x <0 \\
\frac{x^2}{2t},&\quad 0\leq x<yt \\
yx - \frac{y^2t}{2},&\quad yt\leq x\leq \frac 1 y +\frac{yt}{2},\qquad 0<t<2/y^2, \\
1,&\quad \frac 1 y +\frac{yt}{2} < x \leq 4
\end{cases}
\end{align}
and
\begin{align}
\cdf{u_{y,t}} (x) =
\begin{cases}
0,&\quad -1\leq x <0 \\
\frac{x^2}{2t},&\quad 0\leq x\leq \sqrt{2t},\qquad \forall t\geq 2/y^2. \\
1,&\quad x>\sqrt{2t}
\end{cases}
\end{align}
The generalized inverse $\icdf{u_{y,t}}: [0,1] \to \Omega$ has also an explicit expression which reads as
\begin{align}\label{EqExample}
\icdf{u_{y,t}} (s) =
\begin{cases}
-1 &\quad s =0 \\
\sqrt{2ts},&\quad 0< s < \frac{y^2t}{2} \\
\frac{1}{y}\( s + \frac{y^2t}{2} \),&\quad \frac{y^2t}{2}\leq s \leq 1,\qquad 0<t<2/y^2 \\
\end{cases}
\end{align}
and
\begin{align}\label{EqExample2}
\icdf{u_{y,t}} (s) =
\begin{cases}
-1 &\quad s =0 \\
\sqrt{2ts},&\quad 0< s \leq 1,\qquad \forall t\geq 2/y^2.
\end{cases}
\end{align}
We can easily check that the set of solutions
$$
\cM = \{ u_z:= u_{y,t} \; :\; z=(y,t)\in  Z=[0,T]\times \rY \}
$$
is a subset of $\Pr$. As before, we introduce the set $\cT \coloneqq \{ \Log{w}(u) \; : \; u \in \cM \}$, where $w\in \mathcal P_2(\Omega)$, and prove the following result.

\begin{lemma}
 There exists a constant $C>0$ such that for all $n\in \mathbb{N}^*$,
 \begin{equation}\label{eq:res}
 d_n(\mathcal T, L^2(0,1)) \leq C \corr{n^{-21/10}}.
 \end{equation}
\end{lemma}

\begin{proof}
Let us define
$$
\widetilde{\cT}\coloneqq  \{ \icdf{u} \; : \; u \in \cM \} =  \{ \icdf{u(z)} \; : \; z \in \rZ \}. 
$$
Since $\cT$ is a shift by $\icdf{w}$ of $\widetilde \cT$, we have
\corr{
$$
 d_{n+1}(\widetilde \cT, L^2([0,1]))  \leq d_{n}(\cT, L^2([0,1])) \leq 2  d_{n}(\widetilde \cT, L^2([0,1]))
$$
where the second inequality holds thanks to the fact that $\icdf{w}$ is taken in the convex hull of $\cT$.}
Thus, to prove \eqref{eq:res}, it is enough to prove that there exists $C>0$ such that for all $n\in \mathbb{N}^*$, 
\begin{equation}\label{eq:dist}
d_{n}(\widetilde{\mathcal{T}}, L^2(0,1)) \leq C \corr{n^{-21/10}}. 
\end{equation}
Let us prove \eqref{eq:dist}. \corr{To this aim, let us first point out that, for all $(y,t)\in$, the function $g_{y,t}: [0,1] \to \Omega$ defined by
\begin{align}\label{EqExampleg}
g_{y,t} (s) =
\begin{cases}
0 &\quad s =0 \\
\sqrt{2ts},&\quad 0< s < \frac{y^2t}{2} \\
\frac{1}{y}\( s + \frac{y^2t}{2} \),&\quad \frac{y^2t}{2}\leq s \leq 1,\\
\end{cases}
\quad \qquad\mbox{ if } 0<t<2/y^2 ,
\end{align}
and
\begin{align}\label{EqExample2g}
g_{y,t} (s) =
\begin{cases}
0 &\quad s =0 \\
\sqrt{2ts},&\quad 0 < s \leq 1,
\end{cases}
\quad \qquad\mbox{ if } t\geq 2/y^2
\end{align}
are equal almost everywhere to $\icdf{u_{y,t}}$, so that $\icdf{u_{y,t}} = g_{y,t}$ in $L^2([0,1])$.}

\corr{
It is easy to see that $g_{y,t}$ is continuous and that for all $s\in (0,1)$,
\begin{align}\label{EqExampleg}
g'_{y,t} (s) =
\begin{cases}
\sqrt{\frac{t}{2s}},&\quad 0< s < \frac{y^2t}{2} \\
\frac{1}{y},&\quad \frac{y^2t}{2}\leq s \leq 1,\\
\end{cases}
\quad \qquad\mbox{ if } 0<t<2/y^2 ,
\end{align}
and
\begin{align}\label{EqExample2g}
g_{y,t}' (s) =
\begin{cases}
0 &\quad s =0 \\
\sqrt{\frac{t}{2s}},&\quad 0< s \leq 1,
\end{cases}
\quad \qquad\mbox{ if } t\geq 2/y^2\,.
\end{align}
}

\corr{
Let $s_0 \in (0,1)$
and $\eps>0$ be such that 
$I_0\coloneqq[s_0 - \eps/2, s_0 + \eps/2] \subset [0,1]$. We also denote
$$
\forall s\in I_0, \quad  f_1(s) :=1, \quad
 f_2(s) := s,\quad
 f_3(s) := \sqrt{s},
$$
and define $W(I_0):= \mbox{\rm Span}\left\{ f_1, f_2, f_3\right\} \subset L^2(I_0)$. } 

\corr{
Now, let us consider $z\coloneqq(y,t)\in \rZ$ such that  $\frac{y^2t}{2} \in [s_0 - \eps/2, s_0 + \eps/2]$. This implies in particular that $0<t\leq 2/y^2$. 
We denote by $P_{W(I_0)}$ the orthogonal projection from $L^2(I_0)$ onto $W(I_0)$. }

\corr{
We begin by proving two different auxiliary inequalities. On the one hand, it holds that
$$
\left\| g_z - P_{W(I_0)}(g_z|_{I_0})\right\|_{L^2(I_0)} \leq \left\| g_z - \sqrt{2t}f_3 - \frac{1}{y}f_2 + \frac{yt}{2}f_1\right\|_{L^2(I_0)} = \left\| h_z\right\|_{L^2(I_0)},
$$
where $h_z:= g_z - \sqrt{2t}f_3 - \frac{1}{y}f_2 + \frac{yt}{2}f_1$. It then holds that $h_z\left(\frac{y^2t}{2}\right) = 0$ and
\begin{align}\label{EqExampleg}
h'_{z} (s) =
\begin{cases}
- \frac{1}{y},&\quad s_0-\eps/2 \leq s < \frac{y^2t}{2} \\
-\sqrt{\frac{t}{2s}} ,&\quad \frac{y^2t}{2}\leq s \leq s_0+\eps/2.\\
\end{cases}
\end{align}
}

\corr{The function $h_z$ is thus Lipschitz and its Lipschitz constant on the interval $I_0$ is bounded by $\frac{1}{y}$. Since $y \in [1/2,3]$, we obtain that
\begin{equation}\label{EqSimpleEstimate}
\sup_{s\in I_0} | h_z(s) | \leq 2\eps.
\end{equation}}
This implies that
\corr{
\begin{equation}\label{eq:est10}
\|g_z -  P_{W(I_0)}(g_z|_{I_0})\|_{L^2(I_0)}^2 \leq \left\| h_z\right\|_{L^2(I_0)}^2 \leq 4\eps^3,
\end{equation}
so that 
\begin{equation}\label{eq:est1}
\| g_z -P_{W(I_0)}(g_z|_{I_0}) \|_{L^2(I_0)}\leq 2\eps^{3/2}.
\end{equation}
}
\corr{On the other hand,} we make use of the following \corr{Taylor's} inequality: for all maps $f: [0,1] \to \bR$ whose \corr{derivative} is $M-$Lipschitz and all $s,s'\in[0,1]$,
\begin{equation}\label{EqLipschitzEstimate}
|f(s') - f(s) - f'(s) (s'-s) |\leq \frac{M}{2} |s-s' |^2.
\end{equation}
\corr{
Let us define by $j_z:= g_z - \sqrt{2t}f_3$. Then, it holds that $j_z(s) = 0$ if $s_0 - \eps/2\leq s \leq \frac{y^2t}{2}$ and 
\begin{align}\label{EqExamplej}
j'_{z} (s) = \frac{1}{y} -\sqrt{\frac{t}{2s}} ,\quad  \mbox{ for } \frac{y^2t}{2}\leq s \leq s_0 + \eps/2.
\end{align}
As a consequence, $j_z'$ is continuous on $I_0$, differentiable on $I_0 \setminus\left\{ \frac{y^2t}{2}\right\}$ and
\begin{align}\label{EqExamplej2}
j''_{z} (s) =\frac{1}{2\sqrt{2}}\sqrt{t}  s^{-3/2},\quad  \mbox{ for }\frac{y^2t}{2} < s \leq s_0+\eps/2.
\end{align}}
  \corr{Thus, if $t>0$, the Lipschitz constant of $j_z'$ is bounded by $\frac{1}{y^3t}$ on the interval $I_0$.}

\corr{We then apply the inequality \eqref{EqLipschitzEstimate} at the point $s' = \frac{y^2t}{2}$ to obtain that for all $s\in I_0$, 
$$
|j_z(s)| \leq \eps^2 \frac{1}{2 y^3t}.
$$
Therefore, we get the bound
\begin{equation}\label{eq:est20}
\| g_z - P_{W(I_0)}(g_z|_{I_0}) \|^2_{L^2(I_0)} \leq \| j_z \|^2_{L^2(I_0)} \leq  \frac{1}{4y^6t^2} \eps^{5},
\end{equation}
so that
\begin{equation}\label{eq:est2}
\| g_z - P_{W(I_0)}(g_z|_{I_0}) \|_{L^2(I_0)} \leq  \frac{1}{2y^3t} \eps^{5/2}.
\end{equation}}

\medskip

We are now in position to prove \eqref{eq:dist}. Let $\beta >1$ be a constant whose value will be fixed later on. Let $n\in \mathbb{N}^*$ and let us \corr{consider a partition 
of the interval $[0,1]$ into $2n$ intervals $(I_k)_{1\leq k \leq 2n}$ so that the $n$ first intervals (closest to the point $s=0$) are of length $\frac{1}{n^{\beta}}$ and the other $n$ intervals are of length at most 
$\frac{1}{n}$. 
More precisely, we define} for $1\leq k \leq n$, 
$$
x_k\coloneqq \frac{1}{2n^\beta} + (k-1)\frac{1}{n^\beta} \quad \mbox{ and } \quad I_k\coloneqq \left[x_k - \frac{1}{2n^\beta}, x_k + \frac{1}{2n^\beta}\right).
$$
Besides, for all $n+1 \leq k \leq 2n$, we define 
$$
x_k \coloneqq \frac{n}{n^{\beta}} + \frac{1}{2n} + (k-n-1)\frac{1}{n} \quad \mbox{ and } I_k\coloneqq \left[\min\left(1,x_k - \frac{1}{2n}\right),\min\left(1, x_k +\frac{1}{2n}\right)\right).
$$

We then define the space $V_n \subset L^2(0,1)$ as follows:
$$
V_n\coloneqq \mbox{\rm Span}\left\{\charFun_{I_k}(s),\; \charFun_{I_k}(s)s, \; \charFun_{I_k}(s) \sqrt{s}, \; 1\leq k \leq 2n \right\}.  
$$
In other words, $V_n$ is composed of the functions which, on each interval $I_k$, is a linear combination of an affine function and the square root function. The dimension of the space $V_n$ is at most $6n$. We denote by $P_{V_n}$ the orthogonal 
projection \corr{from} $L^2(0,1)$ onto $V_n$. 

It holds that for all $z\coloneqq(y,t)\in\rZ$, 
$$
\left\|\icdf{u(z)} - P_{V_n}(\icdf{u(z)}) \right\|_{L^2(0,1)} \corr{= \left\|g_z - P_{V_n}(g_z) \right\|_{L^2(0,1)}  = \left\|P_{W(I_{k_0})}(g_z|_{I_{k_0}}) - g_z \right\|_{L^2(I_{k_0})}}
$$
where \corr{$1\leq k_0\leq 2n$ is the unique integer between $1$ and $2n$ such that $\frac{y^2t}{2}\in I_{k_0}$.}
On the one hand, if $1\leq k_0 \leq n$, we make use of inequality \eqref{eq:est1} to obtain that 
$$
\corr{\left\|g_z - P_{W(I_{k_0})}(g_z|_{I_{k_0}}) \right\|_{L^2(I_{k_0})} \leq  2 n^{-3\beta/2}.}
$$
On the other hand, if $n+1\leq k_0 \leq 2n$, necessarily $\frac{y^2t}{2} \geq \frac{n}{n^\beta}$, so that $\corr{y^3t\geq 2y n^{1-\beta}}$. We then make use of inequality \eqref{eq:est2} to obtain that 
$$
\corr{\left\|g_z - P_{W(I_{k_0})}(g_z|_{I_{k_0}})\right\|_{L^2(I_{k_0})} \leq   \frac{1}{4y}  n^{-5/2} n^{\beta -1} \leq \frac{1}{2}  n^{-7/2 + \beta}.}
$$
Choosing now $\beta$ such that $\corr{-3\beta/2 = -7/2 + \beta}$, i.e. $\corr{\beta  = \frac{7}{5}}$, we obtain that for all $n\in \mathbb{N}^*$ and $z\in \rZ$, 
$$
\corr{\left\|\icdf{u(z)} - P_{V_n}(\icdf{u(z)}) \right\|_{L^2(0,1)} = \left\|g_z - P_{V_n}(g_z) \right\|_{L^2(0,1)} \leq 2 n^{-21/10}.}
$$
This implies \eqref{eq:dist} and hence the desired result.
\end{proof}

\section{Algorithms for nonlinear reduced modelling in metric spaces}
\label{sec:algo}

In this section, we introduce two methods for building in practice nonlinear reduced models for a set of solutions $\cM$ on general metric spaces. The methods involve a discrete training set $\rZ_{\rm tr}\subset \rZ$ of $N\in \mathbb{N}^*$ parameters and associated $N$ snapshot solutions $\cM_{\rm tr}\subset \cM$. Similarly as before, we denote
$$
\mathcal{T}_{\tr}\coloneqq \Log{w}(\mathcal M_{\tr}) \subset \cT
$$
the image of $\cM_{\tr}$ by the logarithmic map $\Log{w}$, where $w$ is an element of $\mathcal P_2(\Omega)$ to be fixed in a moment. Since our numerical tests are for one-dimensional conservative PDEs in $W_2$, we also instantiate them in this setting.

\subsection{Tangent PCA (tPCA): offline stage}\label{sec:tPCA}
Our first method is based on the so-called Tangent PCA (tPCA) method \corr{(\cite{Fletcher2004PGA,Sommer2014})}. Its definition requires that the metric space $(V,d)$ is embedded with a Riemannian structure.
The tPCA method consists in mapping the manifold to a tangent space and performing a standard PCA on this linearization.

In the offline phase of the tPCA method, we first fix a particular element $w\in\cM$, usually the Fréchet mean
$$
\min_{\tilde w \in V} \frac 1 N \sum_{u\in\cM_{\tr}} d(u, \tilde w)^2,
$$
and then define an inner product $g_w$ on its tangent space $T_wV$. We next consider the PCA (or POD decomposition) on $T_w V$ of the set
$$
\mathcal{T}_{\tr}\coloneqq \Log{w}(\mathcal M_{\tr}),
$$
with respect to the inner product $g_w$. For every $f\in T_wV$, its norm is denoted by $\Vert f \Vert_w = g^{1/2}_w(f,f)$.
There exists an orthonormal family of functions $(f_k)_{k\geq 1}^N$ of $T_w V$ and an orthogonal family 
 $(c_k)_{k\geq 1}$ of functions of $\ell^2(\rZ_\tr)$ norm such that
$$
\Log{w}(u(z))= \sum_{k=1}^N  c_k(z) f_k, \quad \forall z\in {\rZ}_{\rm tr}.
$$
The $k^{th}$ singular value is defined as $\sigma_k\coloneqq \|c_k\|_{\ell^2(\rZ_\tr)}$ and we arrange the indices so that $(\sigma_k)_{k\geq 1}$ is a non-increasing sequence.

In the online phase, we fix $n\in \mathbb{N}^*$ and define $V_n\coloneqq \rm{Span}(f_1, \cdots, f_n)$ and $P_{V_n}$ the orthogonal projection on $T_w V$ with respect to the scalar product $g_w$. For a given $z\in Z$ for which we want to approximate $u(z)$, we consider two possible versions:
\begin{itemize}
\item  \textbf{Projection:} We compute
\begin{equation}\label{eq:deffn}
f_n^{\rm proj}(z)\coloneqq P_{V_n}\Log{w}(u(z)) = \sum_{k=1}^n  c_k(z) f_k,
\end{equation}
and approximate $u(z)$ as 
\begin{equation}
\label{eq:deftPCA}
u_n(z)^{\rm tPCA, proj}\coloneqq \Exp{w}(f_n^{\rm proj}(z)).
\end{equation}
Note that, in fact, the projection \eqref{eq:deffn} cannot be computed online since the computation of the coefficients $c_k(z)$ requires the knowledge of $\Log{w}(u(z))$ (and thus of the snapshot $u(z)$ itself). This motivates the following strategy based on interpolation of the $c_k(z)$.
\item \textbf{Interpolation:} 
Among the possible ways to to make the method applicable online, we propose to compute an interpolation $\overline{c}_k: \rZ \to \mathbb{R}$ such that 
 $$
\overline{c}_k(z) = c_k(z),\quad \forall z\in \rZ_\tr,\, 1\leq k \leq n.
 $$

In the numerical experiments, to avoid stability problems,
we restrict the interpolation to neighboring parameters of the target parameter $z$ which belong to the training set $\rZ_\tr$
Specifically, for a given fixed tolerance $\tau>0$, we find the set $\cN_{\tau}(z, \rZ_\tr)$ of parameters of $\rZ_\tr$ whose $\ell^2$-distance to the current $z$ is at most $\tau$, that is,
$$
\cN_{\tau}(z, \rZ_\tr) \coloneqq \{  \tilde z \in \rZ_\tr\: :\; || z - \tilde z ||_{\ell^2(\bR^p)} \leq \tau \}.
$$
Then, for $k=1,\dots,n$, we build a local interpolator $\bar c^{z,\tau}_k$ that satisfies the local interpolating conditions\footnote{\corr{In our code, we use a local multiquadric radial basis interpolator.}}
$$
\bar c^{z,\tau}_k (z) = c_k (z),\quad \forall z \in \cN_{\tau}(z, \rZ_\tr).
$$
With the local interpolators $\bar c^{z,\tau}_k$, we now compute
\begin{equation}
\label{eq:deffninterp}
f_n^{\rm interp}(z)\coloneqq \sum_{k=1}^n  \bar c^{z,\tau}_k(z) f_k.
\end{equation}
which is an online computable approximation of $\Log{w}(u(z))$. 
Finally, we approximate $u(z)$ with

\begin{equation}
\label{eq:deftPCAinterp}
u_n(z)^{\rm tPCA, interp}\coloneqq \Exp{w}(f_n^{\rm interp}(z)).
\end{equation}
\end{itemize}
Before continuing, several comments are in order:
\begin{itemize}
\item In the online phase, it is necessary to compute exponential maps. However, in full generality, their computation may be expensive since it may require to solve a problem in the full space and not only in a reduced space. This important issue is mitigated in our numerical examples because the exponential map in $W_2$  in one space dimension has an explicit expression. The development of efficient surrogates for the exponential mapping is a topic by itself which we will address in future works.
\item \corr{The approach involving the interpolation of the coefficients is purely \emph{data-driven} in the sense that the online phase does not require solving the original PDE in any reduced space/manifold. This is in contrast to traditional projection-based reduction approaches.}
\item Note that $u^{\rm tPCA, proj}_n(z)$ or $u^{\rm tPCA, interp}_n(z)$ are not always properly defined through \eqref{eq:deftPCA} or \eqref{eq:deftPCAinterp} since it is required that $f^{\rm proj}_n(z)$ or $f^{\rm interp}_n(z)$ belong to $\mathcal{O}_w$, the definition domain of the map $\Log{w}$. Since there is a priori no guarantee that this is always the case, the approach does not lead to a fully robust numerical method and is prone to numerical instabilities as we illustrate in our numerical examples. This drawback has been one main motivation to develop the method based on barycenters, which will be stable by construction. We emphasize nonetheless that the tPCA has important optimality properties in terms of its decay of the approximation error as we describe next. This is probably the reason why numerical methods based on tPCA have drawn significant interest in numerous fields like pattern recognition, shape analysis, medical imaging, computer vision \cite{Fletcher2004PGA,Sommer2014}, and, more recently, statistics \cite{HuckemannPCA} and machine learning to study families of histograms and probability densities \cite{CSBCP2017}. We show evidence that this method also carries potential for model reduction of transport dominated systems in Section~\ref{sec:num}. 
\end{itemize}

\paragraph*{\corr{Case of the 1d Wasserstein space:}}
\corr{When} $(V,d) = (\cP_2(\Omega), W_2)$, \corr{recalling formulas \eqref{eq:exp} and \eqref{eq:log} for the Exp and Log maps and formula \eqref{EqFlatReduction} for the $W_2$ distance,} the offline phase of the tPCA method consists in performing the following steps:
\begin{itemize}
 \item Compute $\overline{f}\coloneqq \frac{1}{N}\sum_{u \in \cM_\tr} \icdf{u}$ and $w = \operatorname{icdf}^{-1}\left( \overline{f}\right)$.
 \item Compute 
 $$
 \mathcal T_{\rm tr}\coloneqq \left\{ \icdf{u} - \icdf{w} \,:\, u\in \cM_{\rm tr} \right\} \subset L^2([0,1]). 
 $$
 \item Compute the $n$ first modes of the PCA of the discrete set of functions $\mathcal T_{\rm tr}$ in $L^2([0,1])$ and denote them by $f_1, \cdots, f_n$; 
 \item 
 Similarly as before, there exists an orthonormal family of functions $(f_k)_{k\geq 1}^N$ of $L^2([0,1])$ and an orthogonal family 
 $(c_k)_{k\geq 1}$ of functions of $\ell^2(\rZ_\tr)$ with non-increasing $\ell^2$ norm such that for all $z\in {\rZ}_{\rm tr}$,
$$
\Log{w}(u(z))= \sum_{k=1}^N  c_k(z) f_k, \quad \forall z\in {\rZ}_{\rm tr}.
$$
The $k^{th}$ singular value is defined as $\sigma_k\coloneqq \|c_k\|_{\ell^2(\rZ_\tr)}$ and
$$
c_k(z)\coloneqq \langle \icdf{u(z)} - \icdf{w}, f_k \rangle_{L^2(0,1)}.
$$
For the online phase, if we work with a dimension $n\in\bN^*$ for the reduced space, we store the coefficients $c_k(z)$ for all $z\in\rZ_\tr$ and $1\leq k \leq n$.
 
\end{itemize}
In the online stage, we fix a dimension $n\in\bN^*$ and, for a given target $z\in \rZ$ for which we want to approximate $u(z)$, we perform the following steps:
\begin{itemize}
 \item \textbf{Projection:} Compute for all $1\leq k \leq n$, 
 $$
c_k(z)\coloneqq \langle \icdf{u(z)} - \icdf{w}, f_k \rangle_{L^2([0,1])}
 $$
 and
 $$
 f_n^{\rm proj}(z)\coloneqq \sum_{k=1}^n c_k(z)f_k.
 $$
 The reduced-order model approximation $u_n(z)^{\rm tPCA, proj}$ of $u(z)$ then reads as
\begin{equation}
u_n(z)^{\rm tPCA, proj}\coloneqq \operatorname{icdf}^{-1}\left( \icdf{w} + f_n^{\rm proj}(z))\right), 
\end{equation}
provided that $f^{\rm proj}_n(z)$ belongs to $\mathcal{K}_{\icdf{w}}$. 
 
\item \textbf{Interpolation:} For all $1\leq k\leq n$, from the knowledge of the values $(c_k(z))_{z\in \rZ_{\rm tr}}$ stored in the offline phase, we compute an interpolation $\overline{c}_k: \rZ \to \mathbb{R}$ such that 
 $$
\overline{c}_k(z) = c_k(z),\quad \forall z\in \rZ_\tr,\, 1\leq k \leq n.
 $$
 We can proceed similarly as before and do a local interpolation. 
We then compute 
\begin{equation}
f_n^{\rm interp}(z)\coloneqq \sum_{k=1}^n  \overline{c}_k(z) f_k,
\end{equation}
and approximate $u(z)$ with
\begin{equation}
u_n(z)^{\rm tPCA, interp}\coloneqq\operatorname{icdf}^{-1}\left( \icdf{w} + f_n^{\rm interp}(z))\right), 
\end{equation}
provided that $f_n^{\rm interp}(z)$ belongs to $\mathcal{K}_{\icdf{w}}$. 
\end{itemize}

\paragraph*{Error decay:}
By introducing a notion of resolution of the finite set $\cT_\tr$ with respect to $\cT$ in terms of $\eps$-coverings, we can derive a convergence result on the average error in the original set $\cM$ which is, as defined in \eqref{eq:av},
$$
e^{\rm tPCA, proj}_\av(\cM, W_2, V_n) \coloneqq
\left( \int_{z\in\rZ} W^2_2\left( u(z), u_n(z)^{\rm tPCA, proj} \right) \,\mathrm{d} \mu(z)  \right)^{1/2}
$$
We next recall the precise definition of $\eps$-covering of a set and its covering number and give a convergence result of the error.
\begin{definition}[$\eps$-covering and covering number]
\label{def:covering}
Let $S$ be a subset of $\corr{(\Pr,W_2)}$. An $\eps$-covering of $S$ is a subset $C$ of $S$ if and only if $S\subseteq \cup_{x\in C} B(x,\eps)$, where $B(x,\eps)$ denotes the ball centered at $x$ of radius $\eps$. The covering number of $S$, denoted $N_S(\eps)$, is the minimum cardinality of any $\eps$-covering of $S$.
\end{definition}

\begin{lemma}
\corr{
If $\cT_\tr$ is an $\eps$-covering of $\cT$, and if $f^{\rm proj}_n(z)\in \mathcal{K}_{\icdf{w}}$ for all $z\in \rZ$, then
$$
e^{\rm tPCA, proj}_\av(\cM, W_2, V_n) \leq \eps + \left( \sum_{k> n}^M \sigma_i^2 \right)^{1/2}.
$$
}
\end{lemma}
\begin{proof}
\corr{
Since for any $u\in\cT$, there exists $\tilde u\in \cT_\tr$ such that $\Vert u - \tilde u\Vert_{L^2([0,1])}\leq \eps$, then
$
\Vert u - P_{V_n} u \Vert_{L^2([0,1])} \leq \eps + \Vert \tilde u - P_{V_n} \tilde u \Vert_{L^2([0,1])}
$
. Thus
$$
e^{\rm tPCA, proj}_\av(\cT, L^2([0,1]), V_n) \leq \eps + e^{\rm tPCA, proj}_\av(\cT_\tr, L^2([0,1]), V_n) = \eps + \left( \sum_{k>n}^N \sigma_i^2 \right)^{1/2}.
$$
We conclude by observing that if $f^{\rm proj}_n(z)\in \mathcal{K}_{\icdf{w}}$ for all $z\in \rZ$, then
$$
e^{\rm tPCA, proj}_\av(\cT, L^2([0,1]), V_n) = e^{\rm tPCA, proj}_\av(\cM, W_2, V_n).
$$
}
\end{proof}

\subsection{The barycentric greedy algorithm (gBar)}
\label{sec:gbar}

The potential instabilities of the tPCA method lead us to consider an alternative strategy for the contruction of reduced-order models, based on the use of barycenters~\cite{pennec2018}. The approach that we propose here can be defined for general metric spaces $(V,d)$ which may not be embedded with a Riemannian manifold structure. Contrary to tPCA, it is guaranteed to be stable in the sense that all the steps of the algorithm are well-defined (approximations in $T_wV$ will be in $\cO_w$ by construction). The stability comes at the price of difficulties in connecting theoretically its approximation quality with optimal performance quantities. Thus its quality will be evaluated through numerical examples.

The method relies on a greedy algorithm, and is henceforth refered to hereafter as the \itshape barycentric greedy \normalfont (gBar) method. 
Let $\eps >0$ be a prescribed level of accuracy. The offline phase of the gBar method is an iterative algorithm which can be written as follows: 

\begin{itemize}
 \item \textbf{Initialization:} \normalfont Compute $(z_1,z_2)\in \rZ_{\rm tr}\times \rZ_{\rm tr}$ such that
 $$
 (z_1,z_2) \in \mathop{\rm argmax}_{(\widetilde{z}_1,\widetilde{z}_2) \in \rZ_{\rm tr}\times \rZ_{\rm tr}} d(u(\widetilde{z}_1),u(\widetilde{z}_2))^2,
 $$
 and define $\rU_2\coloneqq \{ u(z_1),u(z_2) \}$. Then compute
 $$
 \Lambda_2(z) \in \mathop{\rm argmin}_{\Lambda_2 \in \Sigma_2} d\left( u(z), \bary(\rU_2, \Lambda_2)\right)^2,\quad \forall z \in \rZ_\tr.
 $$
 \item \textbf{Iteration $n\geq 3$:} 
  Compute $z_{n}\in \rZ_{\rm tr}$ such that
 $$
 z_{n} \in \mathop{\rm argmax}_{\tilde z \in \rZ_{\rm tr}} \mathop{\min}_{b\in \mathcal{B}_{n-1}(\rU_{n-1})} d(u(\tilde z),b)^2.
 $$
and set $\rU_{n}\coloneqq \rU_{n-1} \cup \{ u(z_{n})\}$. Then compute
  $$
 \Lambda_{n}(z) \in \mathop{\rm argmin}_{\Lambda_n \in \Sigma_n} d\left( u(z), \bary(\rU_n, \Lambda_n)\right)^2,\quad \forall z \in \rZ_\tr.
 $$
 The algorithm terminates when
 $$
 \max_{\tilde z\in \rZ_\tr} \mathop{\min}_{b\in \mathcal{B}_{n-1}(\rU_{n-1})} d(u(\tilde{z}),b)^2  = \mathop{\min}_{b\in \mathcal{B}_{n-1}(\rU_{n-1})} d(u(z_n),b)^2< \eps^2.
 $$
\end{itemize}

Note that the gBar algorithm selects via a greedy procedure particular snapshots $\rU_n = \{ u(z_1), \cdots, u(z_n)\}$ in order to approximate as well as possible each element $u(z)\in \mathcal{M}_{\rm tr}$ with its optimal barycenter associated to the family $\rU_n$. The barycentric weights have to be determined via an optimization procedure.


\medskip

Similarly to the tPCA method, we can consider two different versions of the online phase of the gBar algorithm, whose aim is to reconstrct, for a given $n\in \mathbb{N}^*$, 
a reduced-model approximation $u_n^{\rm gBar}(z)$ of $u(z)$ for all $z\in \rZ$. These two different versions consist in the following steps:

\begin{itemize}
 \item \bfseries Projection: \normalfont Let $z\in \rZ$. Compute $\Lambda_n(z) \in \Sigma_n$ a minimizer of 
 $$
 \Lambda_n(z) \in \mathop{\rm argmin}_{\Lambda_n \in \Sigma_n} d(u(z), \bary(\rU_n, \Lambda_n))^2,
 $$
 and choose $u_n^{\rm gBar, proj}(z) \in \bary(\rU_n,\Lambda_n(z))$.
  \item \bfseries Interpolation: \normalfont From the values $\left(\Lambda_n(z)\right)_{z\in \rZ_{\rm tr}}$ which are known from the offline stage, compute an interpolant $\overline{\Lambda}_n: \rZ \to \Sigma_n$ such that 
  $$
  \overline{\Lambda}_n(z) = \Lambda_n(z),\quad \forall z\in \rZ_\tr.
  $$
  For a given $z\in \rZ$, we approximate $u(z)$ with $u_n^{\rm gBar, interp}(z) \in \bary(\rU_n,\overline{\Lambda}_n(z))$.
\end{itemize}

Like for the tPCA method, the only efficient online strategy is the one based on the interpolation of the barycentric coefficients, since the projection method requires the computation of the full solution $u(z)$ for $z\in \rZ$. \corr{Both approaches are purely data-driven and do not involve solving the original PDE in a reduced space or manifold in the online phase.} We compare the quality of both strategies in our numerical tests. 
\corr{
\begin{remark}
Note that for the interpolation of the coefficients in tPCA and gBar, there is no guarantee that the maps connecting parameters to coefficients is smooth. This property seems difficult to establish a priori and it depends on the regularity of Log and Exp and on the specific problem. The regularity in the coefficients may also strongly depend on the choice of the neighbors, hence on the resolution of the training set $\cM_\tr$. The results of Appendix C tend to confirm this fact since they reveal that the approximation error of the procedure is pretty sensitive to the resolution of $\cM_{\tr}$. 
\end{remark}
}

In the particular case where $(V,d) = (\mathcal P_2(\Omega), W_2)$, every step of the greedy barycentric algorithm can be made explicit by means of inverse cumulative distribution functions. The offline phase is:
\begin{itemize}
 \item \textbf{Initialization:} \normalfont Compute $(z_1,z_2)\in \rZ_{\rm tr}\times \rZ_{\rm tr}$ such that
 $$
 (z_1,z_2) \in \mathop{\rm argmax}_{(\widetilde{z}_1,\widetilde{z}_2) \in \rZ_{\rm tr}\times \rZ_{\rm tr}} \left\|\icdf{u(\widetilde{z}_1)} - \icdf{u(\widetilde{z}_2)}\right\|_{L^2(0,1)}^2,
 $$
 and define $\rU_2\coloneqq(u(z_1),u(z_2))$. Compute and store
 \begin{equation}\label{eq:opt1gen}
 \Lambda_2(z)\coloneqq(\lambda_1^2(z), \lambda_2^2(z)) \in \mathop{\rm argmin}_{\Lambda_2\coloneqq(\lambda_1,\lambda_2) \in \Sigma_2} \left\| \icdf{u(z)} - \sum_{k=1}^2 \lambda_k \icdf{u(z_k)}\right\|_{L^2(0,1)}^2,\quad \forall z \in \rZ_{\rm tr}.
 \end{equation}
 \item \textbf{Iteration $n\geq 3$:}
 Given the set of barycentric coefficients
 $$
 \Lambda_{n-1}(z)\coloneqq (\lambda_1^{n-1}(z), \cdots, \lambda_{n-1}^{n-1}(z)),\quad \forall z\in\rZ_\tr.
 $$
 from the previous iteration, find
 \begin{align*}
 z_{n} &  \in \mathop{\rm argmax}_{z \in \rZ_{\rm tr}} \mathop{\min}_{\Lambda_{n-1}\coloneqq(\lambda_1, \cdots, \lambda_{n-1})\in \Sigma_{n-1}} \left\|\icdf{u(z)} - \sum_{k=1}^{n-1} \lambda_k \icdf{u(z_k)}\right\|_{L^2(0,1)}^2\\
 & = \mathop{\rm argmax}_{z \in \rZ_{\rm tr}} \left\|\icdf{u(z)} - \sum_{k=1}^{n-1} \lambda^{n-1}_k(z) \icdf{u(z_k)}\right\|_{L^2(0,1)}^2 \\
 \end{align*}
 and set $\rU_{n}\coloneqq \rU_{n_1}\cup \{ u(z_{n})\}$. Compute and store
 \begin{equation}\label{eq:compcoeffgen}
 \Lambda_{n}(z)\coloneqq(\lambda_1^n(z),\cdots, \lambda_n^n(z))  \in \mathop{\rm argmin}_{\Lambda_n\coloneqq(\lambda_1,\cdots,\lambda_n) \in \Sigma_n}\left\| \icdf{u(z)} - \sum_{k=1}^n \lambda_k \icdf{u(z_k)} \right\|_{L^2(0,1)}^2,\quad \forall z\in \rZ_\tr.
 \end{equation}
 The algorithm terminates if
 $$
 \mathop{\min}_{\Lambda_{n-1}\coloneqq(\lambda_1, \cdots, \lambda_{n-1})\in \Sigma_{n-1}} \left\|\icdf{u(z_n)} - \sum_{k=1}^{n-1} \lambda_k \icdf{u(z_k)}\right\|_{L^2(0,1)}^2 < \eps^2.
 $$
\end{itemize}

For a fixed $n\in \mathbb{N}^*$, the two versions of the online phase of the gBar method read as follows:
\begin{itemize}
 \item \bfseries Projection: \normalfont Given $z\in \rZ$ for which we want to approximate $u(z)$, compute
\begin{equation}\label{eq:projetape}
 \Lambda_n(z) \in \mathop{\rm argmin}_{\Lambda_n\coloneqq(\lambda_1, \cdots, \lambda_n) \in \Sigma_n} \left\| \icdf{u(z)} - \sum_{k=1}^n \lambda_k \icdf{u(z_k)}\right\|_{L^2(0,1)}^2,
 \end{equation}
 and define $u_n^{\rm gBar, proj}(z) = \bary{(\rU_n, \Lambda_n(z))}$. The function $u_n^{\rm gBar, proj}(z)\in \mathcal{P}_2(\Omega)$ can easily be defined through its icdf as
 $$
 \icdf{u_n^{\rm gBar, proj}(z)} = \sum_{k=1}^n \lambda_k(z) \icdf{u(z_k)}.
 $$
  \item \bfseries Interpolation: \normalfont From the known values $\left(\Lambda_n(z)\right)_{z\in \rZ_{\rm tr}}$, compute an interpolation 
  $\overline{\Lambda}_n: \rZ \to \Sigma_n$ such that
  $$
  \overline{\Lambda}_n(z) = \Lambda_n(z),\quad \forall z\in \rZ_\tr.
  $$
  For a given target $u(z)$ with $z\in \rZ$, we approximate with $u_n^{\rm gBar, interp}(z)= \bary(\rU_n,\overline{\Lambda}_n(z))$ which is the function of  $\Pr$ such that
 $$
 \icdf{u_n^{\rm gBar, interp}(z)} = \sum_{k=1}^n \overline{\lambda}_k(z) \icdf{u(z_k)}, 
 $$
 where $\overline{\Lambda}_n(z) = (\overline{\lambda}_1(z), \cdots, \overline{\lambda}_n(z))$. 
\end{itemize}

Notice that \eqref{eq:opt1gen} and \eqref{eq:compcoeffgen} amounts to solving a simple convex quadratic programming problem.

\section{Numerical cost}
\label{sec:cost}
We give simple estimates on the numerical complexity of the tPCA and gBar methods in terms of the number of operations. We consider the versions where we do interpolation instead of projection. For this, we introduce some perliminary notation. Let $\cost(u(z))$ be the numerical cost to compute a given snapshot $u(z)=u(t,y)$. \corr{Let us assume that} the full-order PDE discretization has a uniform spatial mesh with $\cN$ degrees of freedom and uses a simple implicit Euler time intergration with time step $\delta t$. \corr{If the equation is linear, then} $\cost(u(z))\sim \cN^3 t /\delta t$ with a direct linear system solver, or $\cost(u(z))\sim k \cN^2 t /\delta t$ with an iterative solver requiring $k$ iterations to converge. \corr{In the case where the equation is nonlinear, then each time step requires the solution of a nonlinear system, which can be dealt with a Newton-type method. If $Q$ iterations are performed, the overall complexity is of order $\cost(u(z))\sim Q \cN^3 t /\delta t$ with a direct linear system solver, or $\cost(u(z))\sim Q k \cN^2 t /\delta t$.}

Let $\cost(\Log{w})$ and $\cost(\Exp{w})$ be the cost of computing the logarithmic and exponential maps.

\paragraph*{tPCA:} In the offline phase, we have to compute $N$ snapshots, compute their logarithmic images  to get $\cT_\tr$, and perform a PCA on $\cT_\tr$. Thus $\cost_{\text{offline}}^{\text{tPCA}} = \sum_{z\in\rZ_\tr} \cost(u(z)) + N \cost(\Log{w})  + \cost^{\text{PCA}}$. Counting the cost of computing the covariance matrix and the eigenvalue computation, we have $\cost^{\text{PCA}}\sim N \cN^2 + \cN^3$. Therefore
$$
\cost_{\text{offline}}^{\text{tPCA}} \sim N( \cN^3 T /\delta t + \cost(\Log{w} )) + N \cN^2 + \cN^3.
$$
For a given $z\in\rZ$, to approximate $u(z)$ with $u_n(z)^{\rm tPCA, interp}$ in the online phase, we have to compute an interpolant and an exponential mapping. If we do the local interpolation of each coefficient $c_k(z)$ with $r$ neighbors, the cost of the interpolation is of order $r^3+p\ln (N)$, where $r^3$ is the cost to solve the final linear system and $p\ln(N)$ is the cost of finding the $r$ nearest neighbors with state of the art algorithms like the ball tree. As a result
$$
\cost(u_n(z)^{\rm tPCA, interp}) \sim  n(r^3+p\ln (N)) + \cost(\Exp{w}),
$$
This cost has to be compared to $\cost(u(z))\sim \cN^3 t /\delta t$, the cost of computing $u(z)$ with the full order model. We clearly see the critical role that the cost of the exponential mapping plays in the efficiency of the reduction method. As already brought up in section \ref{sec:tPCA}, $\cost(\Exp{w})$ can in general be expensive since it may require to solve a problem in the full space.  In the case of $W_2$  in one space dimension, the cost is strongly reduced since the exponential map has an explicit expression. The problem of building good surrogates of the exponential mapping is a topic by itself and its treatment is deferred to future works.

\paragraph*{gBar:} If we perform $n$ steps in the barycentric greedy algorithm, the offline cost scales like 
$$
\cost_{\text{offline}}^{\text{gBar}} \sim N\left(\cN^3 T /\delta t + \cost(\Log{w} ) + n \cost_n(\text{best Bar})\right)
$$
where $\cost_n(\text{best Bar})$ is the cost of computing the best barycenter of a function $u$ with respect to a set $\rU_n$ of $n$ functions. It is equal to the complexity of solving the cone quadratic optimization problem \eqref{eq:minbary} and is typically proportional to $n^2$. In the online phase, the cost to approximate a given target $u(z)$ is
$$
\cost(u_n(z)^{\rm tPCA, interp}) \sim  r^3+p\ln (N) + \cost(\Exp{w}).
$$
Like before, the cost $\cost(\Exp{w})$ is the critical part for the efficiency.

\section{Numerical examples}
\label{sec:num}
As a support for our tests, we consider four different conservative PDEs:
\begin{itemize}
\item The above discussed inviscid Burgers' equation for which we have explicit expressions of the solutions and icdf (see section \ref{sec:inviscid}).
\item The version with viscosity of the previous Burgers' equation.
\item A Camassa Holm equation.
\item A Korteveg de Vries equation.
\end{itemize}
For each PDE, we compare the performance of the four following model reduction methods:
\begin{itemize}
\item The classical PCA method in $L^2$,
\item The tangent PCA method (with projection) in $W_2$,
\item The gBar method (with interpolation and projection) in $W_2$.
\end{itemize}
The performance is measured in terms of the average and worst case approximation error of a set on a discrete test set of 500 functions. \corr{Each test set is different from the training set $\cM_{\tr}$. The training set is composed of randomly generated snapshots. For every example, the number of training snapshots is $\#\cM_{\tr} = 5.10^3$. The size and precise selection of the training snapshots defines the resolution of $\cM_\tr$ since it fixes the smallest value $\eps_\tr$ such that $\cM_{\tr}$ is an $\eps_\tr$-covering of $\cM$ (see Definition \ref{def:covering}). One difficulty is that it is hard to estimate $\eps_\tr$ in practice. Also, for a given target resolution $\eps>0$, the covering number $N_S(\eps)$ becomes potentially very large for small $\eps$ as the parameter dimension becomes large due to the curse of dimensionality. Recent theoretical results obtained in [8] show that in certain relevant PDE classes, $\eps$-coverings can be replaced by random training sets of smaller cardinality in the reduced basis greedy algorithm \cite{CDD2018}. One interesting direction for further research is to develop similar ideas in the context of the present PDEs and the present algorithms. This is however beyond the scope of this paper so we limit ourselves to illustrate numerically the impact of the size $\#\cM_{\tr}$ in Appendix \ref{appendix:training-size}.}

\corr{
In addition to the error study, we also provide run time statistics but only for the case of the viscous Burgers' equation since it is the only example that involves a high-fidelity solver. In the case of the inviscid Burgers' equation and KdV, the exact solutions can be explicitly written down with formulas so we did not use a solver to which we can compare ourselves to. In the case of the Camassa Holm equation, the solution was nearly analytic too and we could not consider its numerical solution as a representative example which involves a solver.
}

The code to reproduce the numerical results is available online at:
\begin{center}
\href{https://github.com/olga-mula/2019-RBM-metric-spaces}{https://github.com/olga-mula/2019-RBM-metric-spaces}
\end{center}
For each PDE example, we also provide reconstruction videos of a full propagation on the same link.

\newcommand{\plotdecaysvd}[1]{
\begin{figure}[htbp!]
\centering
\includegraphics[width=0.45\textwidth]{img/#1/decay-err-training.pdf}
\caption{#1 equation: convergence of average errors in training set.}
\label{fig:decay-svd-#1}
\end{figure}
}
\newcommand{\plotdecaysvdall}[4]{
\begin{figure}[htbp!]
\centering
\begin{subfigure}[b]{0.4\textwidth}
\includegraphics[width=\textwidth]{img/#1/decay-err-training.pdf}
\caption{#1 equation.}
\label{fig:decay-svd-#1}
\end{subfigure}
~
\begin{subfigure}[b]{0.4\textwidth}
\includegraphics[width=\textwidth]{img/#2/decay-err-training.pdf}
\caption{#2 equation.}
\label{fig:decay-svd-#2}
\end{subfigure}
~
\begin{subfigure}[b]{0.4\textwidth}
\includegraphics[width=\textwidth]{img/#3/decay-err-training.pdf}
\caption{#3 equation.}
\label{fig:decay-svd-#3}
\end{subfigure}
~
\begin{subfigure}[b]{0.4\textwidth}
\includegraphics[width=\textwidth]{img/#4/decay-err-training.pdf}
\caption{#4 equation.}
\label{fig:decay-svd-#4}
\end{subfigure}
\caption{Convergence of average errors in training set.}
\end{figure}
}
\newcommand{\ploterror}[2]{
\begin{figure}[htbp!]
\includegraphics[width=0.45\textwidth]{img/#2/error/#1-natural-norms.pdf}
\includegraphics[width=0.45\textwidth]{img/#2/error/#1-Hminus1.pdf}
\caption{#2 equation: convergence errors in #1 sense (left: natural norms, right: $H^{-1}$ norm).}
\label{fig:error-#1-#2}
\end{figure}
}

\newcommand{\ploterrorall}[1]{
\begin{figure}[htbp!]
\centering
\includegraphics[width=0.4\textwidth]{img/#1/error/av-natural-norms.pdf}\quad
\includegraphics[width=0.4\textwidth]{img/#1/error/av-Hminus1.pdf}\\
\includegraphics[width=0.4\textwidth]{img/#1/error/wc-natural-norms.pdf}\quad
\includegraphics[width=0.4\textwidth]{img/#1/error/wc-Hminus1.pdf}
\caption{Errors on $\cMtest$ for the #1 equation. Top figures: average error. Bottom figures: worst case error. Left: natural norms. Right: $H^{-1}$ norm.}
\label{fig:error-#1}
\end{figure}
}

\newcommand{\plotfun}[3]{
\begin{figure}[htbp!]
\centering
\includegraphics[width=0.4\textwidth]{img/#1/n-#2/fun-exact.pdf}
\\
\includegraphics[width=0.4\textwidth]{img/#1/n-#2/fun-PCA.pdf}
\includegraphics[width=0.4\textwidth]{img/#1/n-#3/fun-PCA.pdf}
\\
\includegraphics[width=0.4\textwidth]{img/#1/n-#2/fun-tPCA.pdf}
\includegraphics[width=0.4\textwidth]{img/#1/n-#3/fun-tPCA.pdf}
\\
\includegraphics[width=0.4\textwidth]{img/#1/n-#2/fun-bary.pdf}
\includegraphics[width=0.4\textwidth]{img/#1/n-#3/fun-bary.pdf}
\caption{#1 equation: Reconstruction of a function with $n=#2$ (left) and $n=#3$ (right). Black: exact function. Red: PCA. Green: tPCA. Blue: Barycenter.}
\label{fig:plotfun-#1}
\end{figure}
}


\newcommand{\spike}[2]{
\begin{figure}[htbp!]
\centering
\includegraphics[width=0.45\textwidth]{img/#1/n-#2/fun-exact.pdf}
\includegraphics[width=0.45\textwidth]{img/#1/n-#2/fun-tPCA.pdf}
\caption{Instabilities of tPCA. Left: Target function. Right: Approximation with tPCA (note the spike).}
\label{fig:spike-#1}
\end{figure}
}


\subsection{Inviscid Burgers' equation}
\label{sec:burgers-num}
We consider the same parametric PDE as the one of section \ref{sec:inviscid} and focus first on the average error decay of PCA and tPCA in the training set $\cM_\tr$, denoted by $e_\av(\cM_\tr, L^2, V_n^{\PCA})$ and $e_\av(\cM_\tr, W_2, V_n^{\tPCA})$.
Figure \ref{fig:decay-svd-Burgers} shows that the error with tPCA decreases much faster than the one with PCA. This is connected to the fact that the decay of the $n$-width $\delta_n(\cM, L^2(\Omega))$ and associated singular values is much slower than $\delta_n(\cT, L^2([0,1]))$, which is the width exploited in tPCA (see sections \ref{sec:inviscid} and \ref{sec:tPCA}).


\plotdecaysvdall{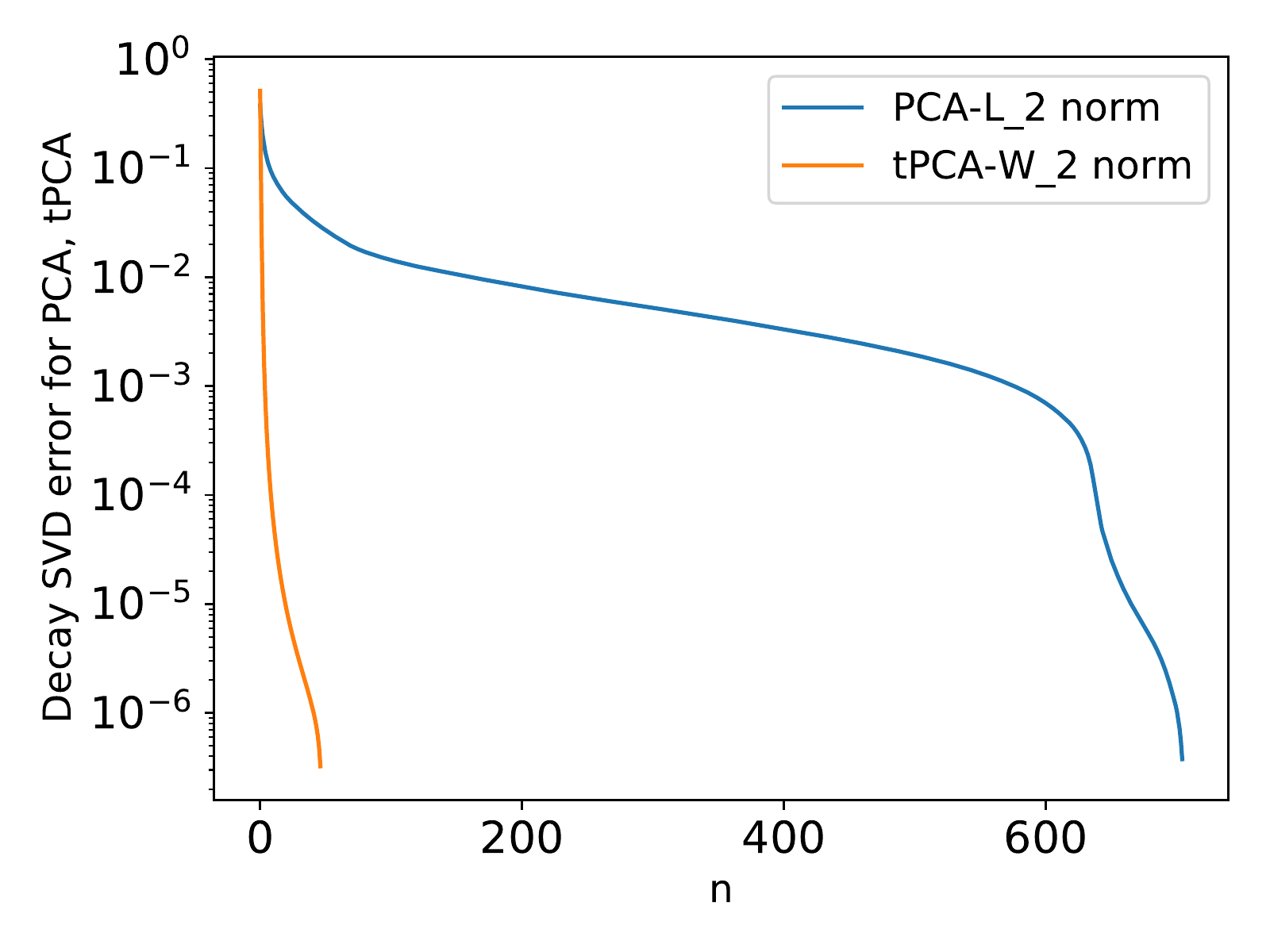}{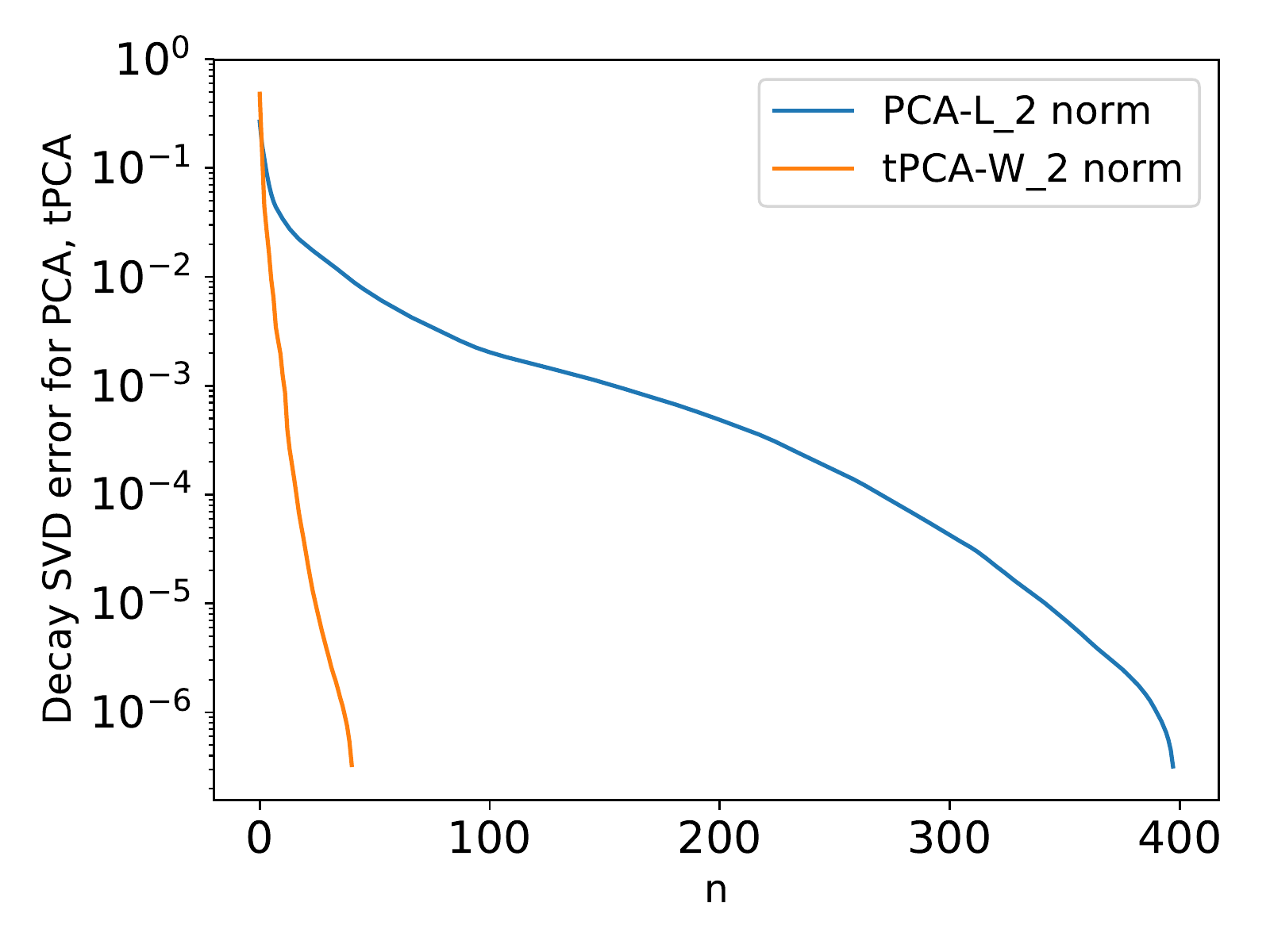}{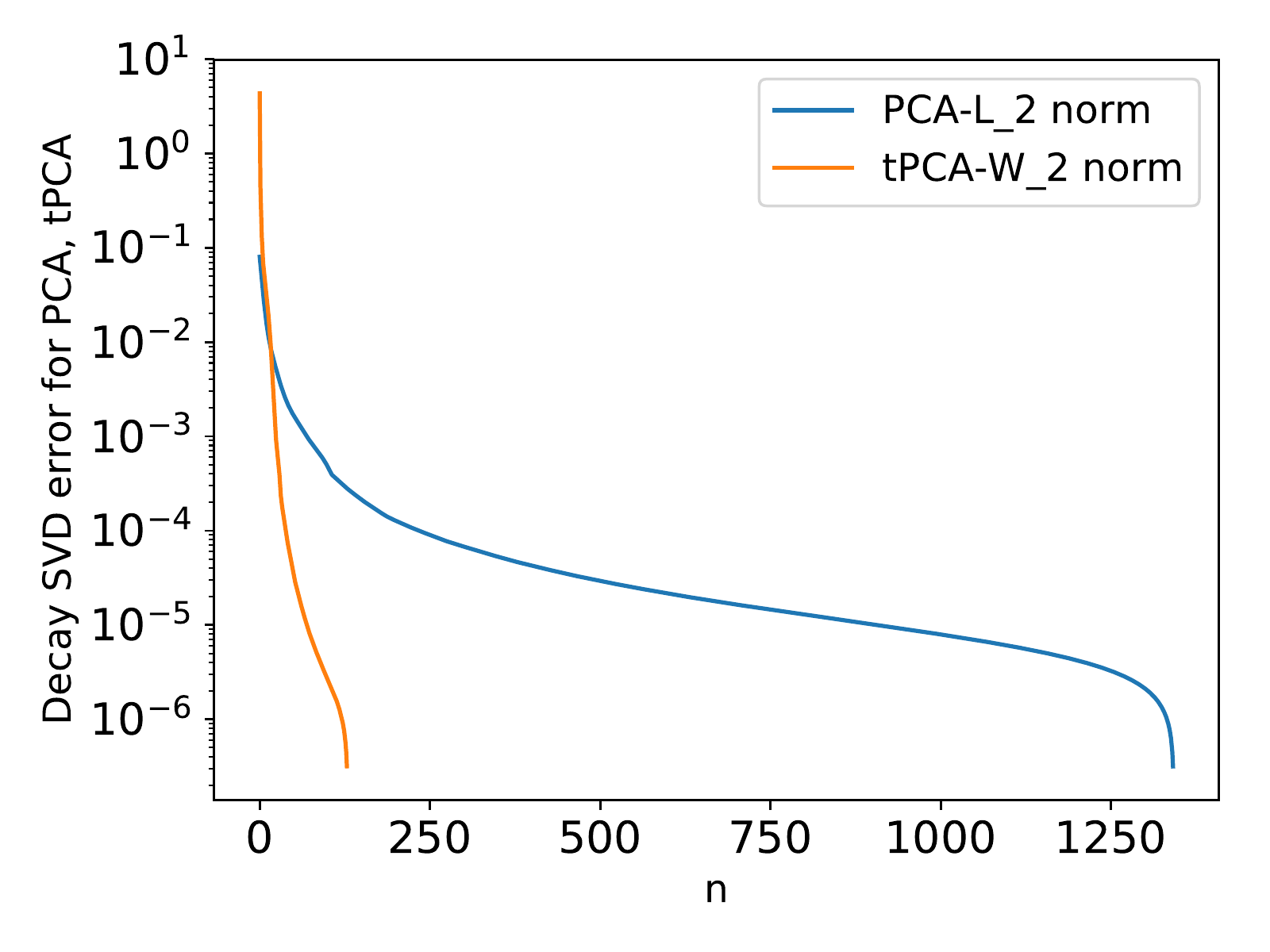}{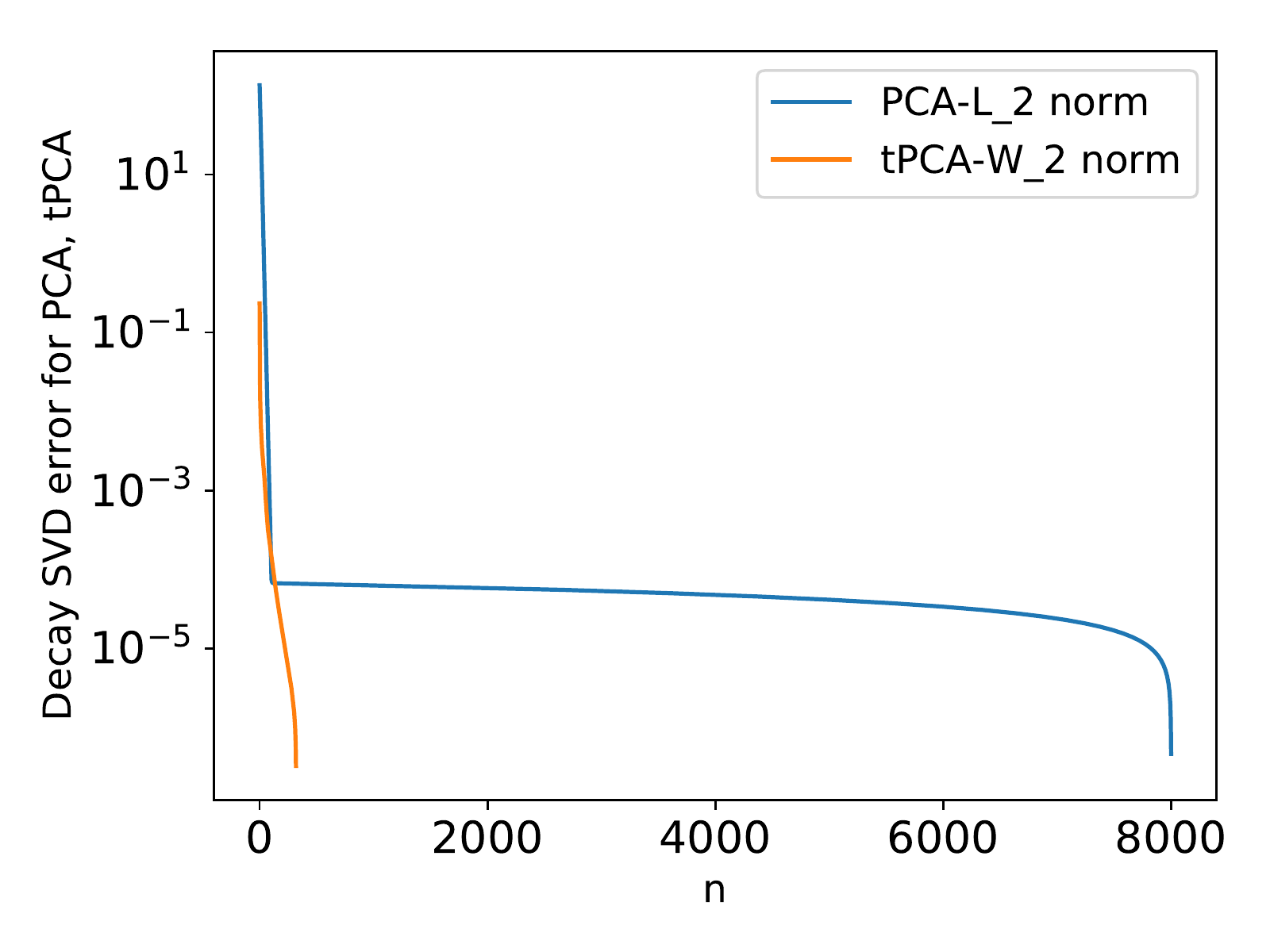}

Figure \ref{fig:error-Burgers} gives the errors in average and worst case over $\cMtest$. The plots on the left measure the errors in the natural norm of each method, that is, $L^2$ for PCA and $W_2$ for the rest.
Note first of all the dramatic difference in the behavior of the error between PCA (which is very slow) and the rest of the approaches (which is much faster). Also, we can see that tPCA presents a faster decay rate compared to the approach with barycenters. 
This may lead to think that tPCA is a better choice for the present context but to confirm this it is necessary to mesure errors in a common metric which is ``fair'' for all approaches and which also quantifies the potential numerical instabilities of tPCA.
Since we are looking for metrics that quantify the quality of transport rather than spatial averages, we discard the $L^2$ metric in favor of the $H^{-1}$ metric which can be seen as a relaxation of the $W_2$ distance. \corr{The space $H^{-1}$ is taken here as the dual of $H^{1}_0$ and its norm computed accordingly.}

The plots on the right in Figure \ref{fig:error-Burgers} measure the errors in this norm.
We again observe the superiority of tPCA and the barycenters' approach with respect to the PCA. The plateau that tPCA and the approach with barycenters reach at a value around $10^{-3}$ is due to the fact that the $H^{-1}$ requires to invert a Laplace operator. For this, we used in our case a discretization with P1 finite elements on a spatial mesh of size $h\approx 5.10^{-4}$.
So the plateau is due to this discretization error and not to the fact that the approximation errors do not tend to 0.

\ploterrorall{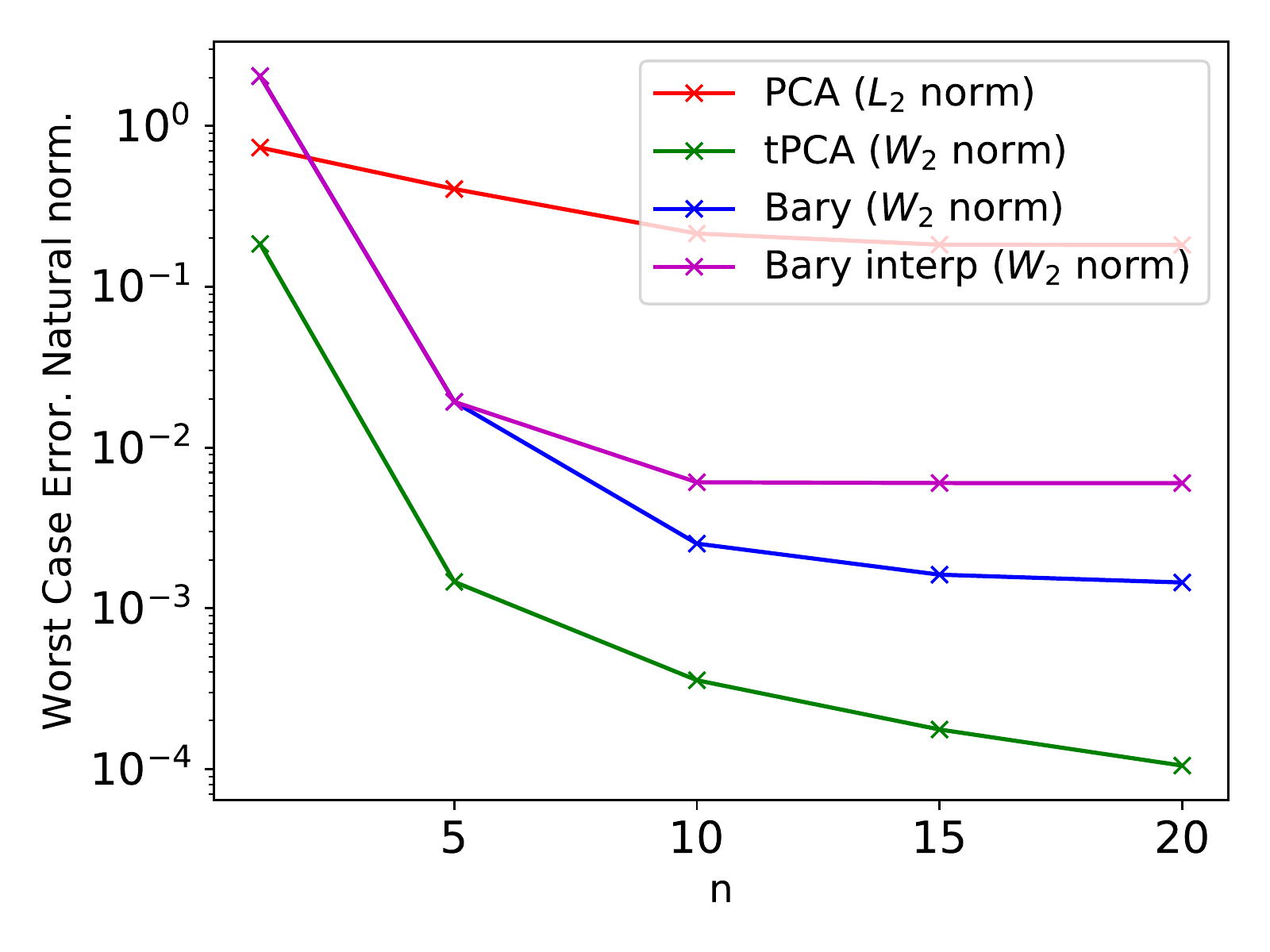}

In Figure \ref{fig:plotfun-Burgers} of Appendix \ref{appendix:plotfun}, we give plots of the reconstruction of one snapshot with all the methods. The PCA approach presents very strong oscillations and fails to capture the sharp-edged form of the snapshot. tPCA and the approach with barycenters give an approximation of higher quality in terms of the ``shape'' of the reconstructed function. The tPCA approximation presents unnatural ``spikes'' which are due to the instabilities described in section \ref{sec:tPCA}. We also provide videos of the reconstruction of a full evolution with our methods on the link above.

\subsection{Viscous Burgers' equation}
We consider the same problem as before but add a viscosity term $\nu$ that ranges in $[5.10^{-5}, 0.1]$. The equation is then
\begin{align}
\label{eq:burgers}
\rho_t + \frac 1 2 (\rho^2)_x - \nu \partial_x^2 \rho=0,
\quad
\rho(0,x,y) =
\begin{cases}
&0,\quad -3\leq x <0 \\
&y,\quad 0\leq x < \frac 1 y \\
&0,\quad  \frac 1 y \leq x \leq 5,
\end{cases}
\end{align}
where, like before, the parameter $y\in [1/2,3]$ and $(t,x)\in[0,T]\times \Omega=[0,3]\times [-3,5]$ (the space interval has slightly been enlarged). The parameter domain is here
$$
\rZ = \{ (t, y, \nu) \in [0,3]\times [0.5, 3]\times [5.10^{-5}, 0.1] \}.
$$
We present the results following the same lines as in the previous example. Figure \ref{fig:decay-svd-ViscousBurgers} shows the decay of the error of the PCA and tPCA methods in the average sense and for the functions used in the training phase. Like before, the error decays dramatically faster in tPCA than in PCA.


Next, Figure \ref{fig:error-ViscousBurgers} gives the errors in average and worst case sense for the test set $\cMtest$. If we first examine the errors in the natural norms (plots on the left), it appears that the errors in tPCA do not seem to decay significantly faster than in PCA. Also, the approach with barycenters does not seem to give a very good performance and seems to perform worse than PCA. However, when we examine the errors in the unified $H^{-1}$ metric, we see that all the nonlinear methods are clearly outperforming PCA. This is more in accordance with what we visually observe when we examine the reconstructed functions given by each method (see Figure \ref{fig:plotfun-ViscousBurgers} of Appendix \ref{appendix:plotfun}). Like before, the approximation with PCA has unnatural oscillations. Note that in this particular example the tPCA presents a sharp unnatural spike at the propagation front, due to the above discussed stability issues of this method. This is in contrast to the approach with barycenters which does not suffer from this issue at the cost of slightly degrading the final approximation quality. Like for the other examples, the reader may watch videos of the reconstruction on the link above.

\ploterrorall{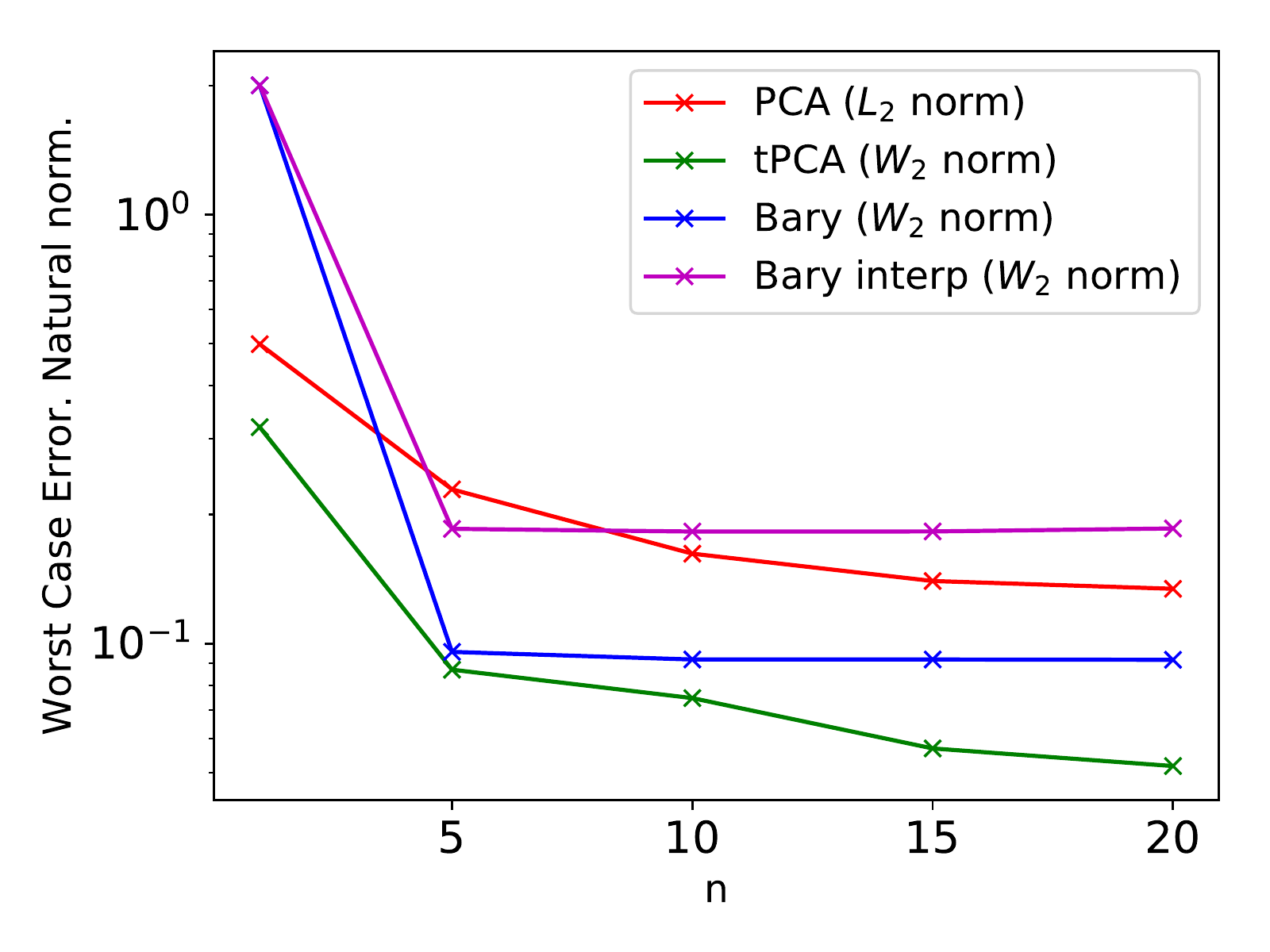}

\corr{
We next provide run time statistics for this test case. For any $u\in \cM_\tr$, let $r_{\text{HF}}(u)$, $r_{\PCA}(u)$, $r_{\tPCA}(u)$, $r_{\gBar}(u)$ and $r_{\gBar}^{\text{interp}}(u)$ be the respective run times of the high-fidelity solver, and of the $\PCA$, $\tPCA$, $\gBar$ and $\gBar$ with interpolation methods. The high-fidelity solver uses an explicit piecewise linear finite-volume method to evaluate the advective flux and then discretize the diffusion part implicitly (Crank-Nicolson) with the advective piece as a source to update in time. The resulting discretization is second-order in space and time. For each dynamic, the time step $\delta t$ is fixed to be sufficiently small in order to satisfy a CFL condition. Figure \ref{fig:runtime-stats} shows, as a function of the reduced dimension $n$, the average and the median of the ratios between the run time of a given method and the run time of the high-fidelity computation,
$$
R^*_{\av} = \frac{1}{\# \cM_\tr}\sum_{u\in \cM_\tr} \frac{r_{*}(u)}{r_{\text{HF}}(u)}
\quad\text{and}\quad
R^*_{\text{median}} = \text{median} \left\lbrace \frac{r_{*}(u)}{r_{\text{HF}}(u)} \, : \, u \in \cM_\tr \right\rbrace .
$$
The $*$ symbol denotes all the previous methods. The figure shows that the run time is reduced by a factor of about $100$ in average and of about $500$ in the median for all the methods. We observe that the classical PCA is slightly faster than tPCA and the gBar algorithm. As discussed in Section \ref{sec:cost}, this is essentially due to the fact that we need to compute exponential maps for the latter methods. We may also note that the run time remains essentially with $n$: we think that this is due to the fact that $n$ is pretty small and expect a mild increase for larger values of $n$. One last final observation for these plots is to remark that the \gBar method with interpolation performs almost identically than the one with interpolation.
}

\begin{figure}[h!]
\includegraphics[width=0.45\textwidth]{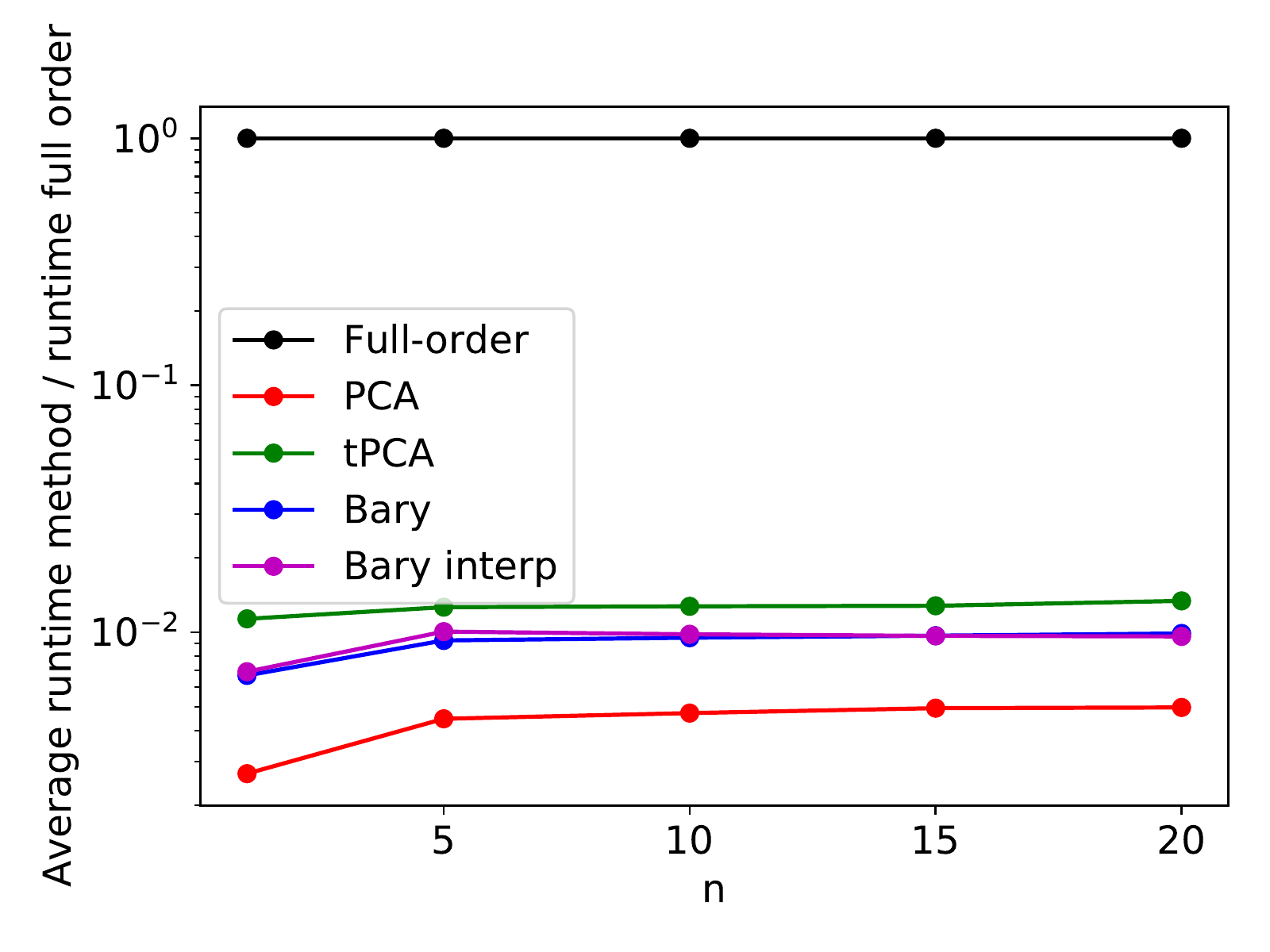}
\includegraphics[width=0.45\textwidth]{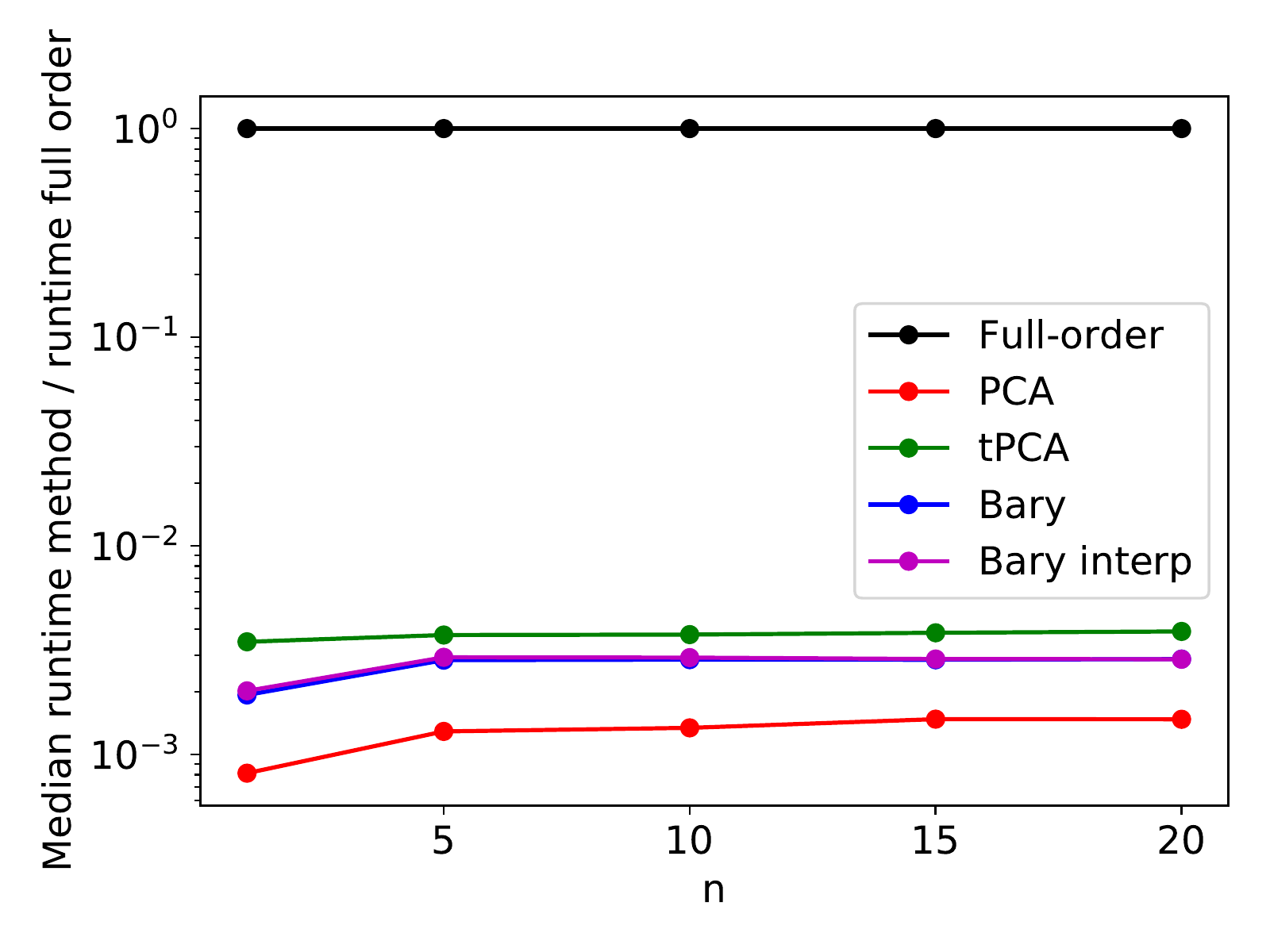}
\caption{Run times as a function of the dimension $n$: Average $R^*_{\av}$ (left plot) and median $R^*_{\text{median}}$ (right plot).} 
\label{fig:runtime-stats}
\end{figure}

\corr{
Since the time variable is treated as a parameter in our approach, it is interesting to compare run times with respect to $t$ since we expect that the high-fidely method will be faster for smaller times. Figure \ref{fig:runtime-stats-t} shows the run times to compute each snapshot of $\cM_\tr$ as a function of its corresponding parameter $t$ (we take $n=20$). We observe that the reduced models are significantly faster, even for small values of $t$. As expected, the difference grows as $t$ increases.
}

\begin{figure}[h!]
\centering
\includegraphics[width=0.45\textwidth]{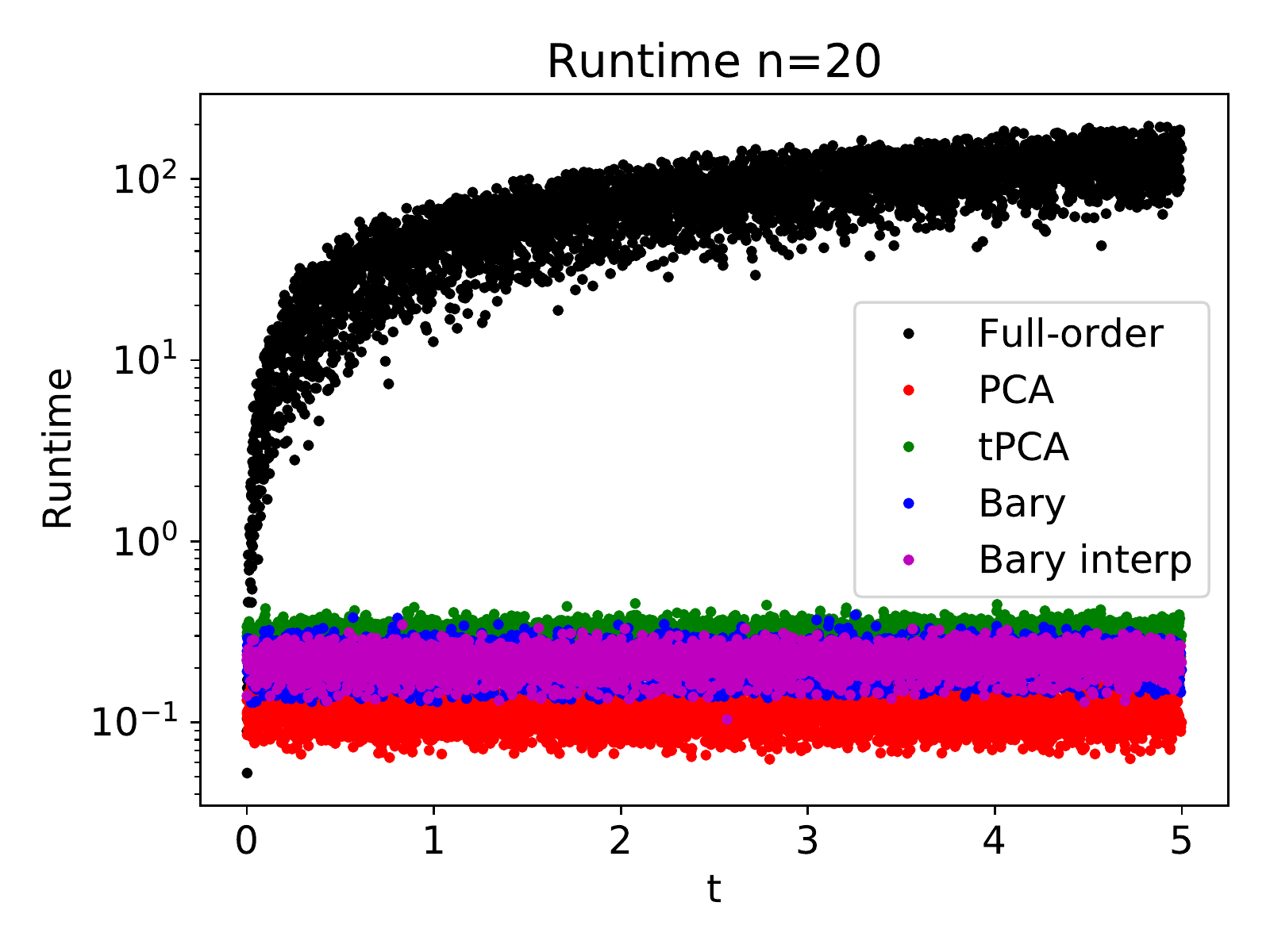}
\caption{Run times to compute each snapshot of $\cM_\tr$ as a function of the parameter $t$ ($n=20$).}
\label{fig:runtime-stats-t}
\end{figure}

\subsection{The Camassa-Holm equation}
We next consider the dispersionless Camassa-Holm equation which, for $\alpha>0$ given, consists in finding the solution $\rho$ of
\begin{equation} 
\partial_t m + \rho m_x + 2 m \partial_x \rho = 0,\quad \text{with } m = \rho - \alpha^2 \partial_{xx} \rho.
\end{equation}
The equation admits peakon solutions of the form
$$
\rho(t,x) = \frac 1 2 \sum_{i=1}^N p_i(t) e^{-|x - q_i(t)| /\alpha},
$$
where the evolution of the set of peakon parameters $p_i(t)$ and $q_i(t)$ satisfies the canonical Hamiltonian dynamics
\begin{equation}
\dot q_i(t) = \frac{\partial h_N}{\partial p_i}
\quad\text{and}\quad
\dot p_i(t) = - \frac{\partial h_N}{\partial q_i}
\end{equation}
for $i=1,\dots, N$, with Hamiltonian given by
\begin{equation}
h_N = \frac 1 4 \sum_{i,j=1}^N p_i p_j e^{-|q_i - q_j|/\alpha}.
\end{equation}
For the parametrized model considered in our tests, we set $N=2$ and $\alpha = 1$. The initial values of the peakons parameters are $p_1(0)=0.2$, $p_2(0)=0.8$, $q_1(0)\in [-2, 2]$ and $q_2(0)=-5$. The parameter domain is then
$$
\rZ = \{ (t, q_1(0)) \in [0,40]\times [-2,2]  \}.
$$
Figure \ref{fig:decay-svd-CH} shows the decay of the error of the PCA and tPCA methods in the average sense and for the functions used in the training phase. Like before, the error decays dramatically faster in tPCA than in PCA.


Figure \ref{fig:error-CH} gives the errors in average and worst case sense for the test set $\cMtest$. Very similar observations like for the case of Viscous Burgers' hold: in the natural norms, the errors in tPCA do not seem to decay significantly faster than in PCA. The approach with barycenters does not seem to give a very good performance and seems to perform worse than PCA. Very different conclusions can be drawn if we consider the unified $H^{-1}$ metric, where we see that all the nonlinear methods are clearly outperforming PCA. Like previously, this is confirmed visually (see Figure \ref{fig:plotfun-CH} of Appendix \ref{appendix:plotfun}).

\ploterrorall{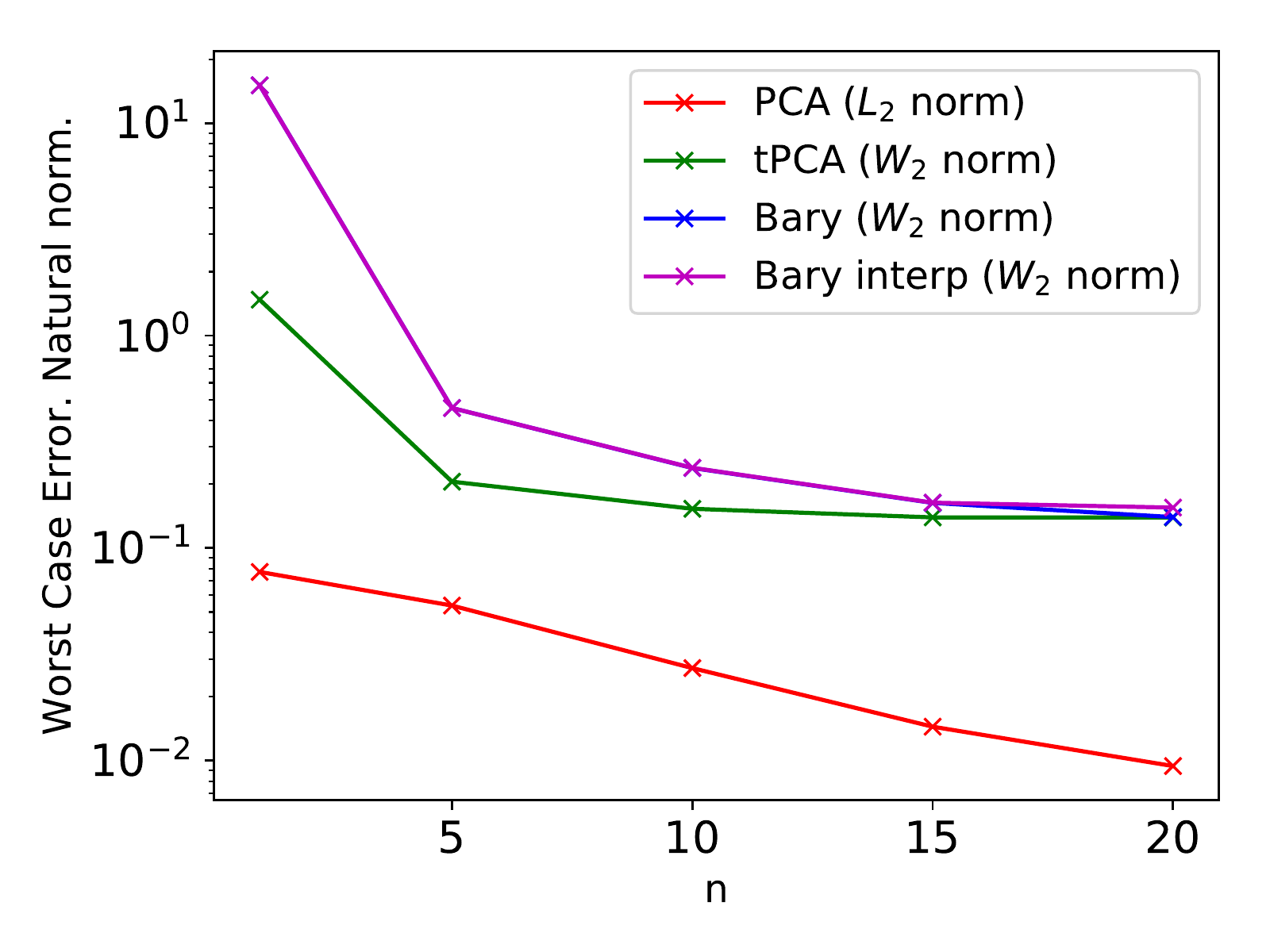}

\subsection{Two-soliton solution of the Korteweg-de-Vries equation}
The previous examples give solid numerical evidence that tPCA and our approach with barycenters are an efficient strategy to build reduced models of certain classes of transport dominated conservative PDEs. We finish the section on numerical results by presenting a PDE which poses certain challenges to the present methodology. This will serve as a transition that will lead us to the end of the paper where we present possible extensions to the present approach.

We consider a two soliton solution of the Korteweg-de-Vries equation (KdV) which, expressed in normalized units, reads for all $x\in\bR$,
\begin{align*}
\partial_t \rho + 6\rho \partial_x \rho + \partial_x^3 \rho = 0.
\end{align*}
The equation admits a general $2$--soliton solution
\begin{eqnarray}
\rho(x,t) = -2 \partial^2_x \log \det(I + A(x,t)),
\end{eqnarray}
where $A(x,t) \in \bR^{2\times 2}$ is the interaction matrix whose components $a_{i,j}$ are
$$
a_{i,j}(x,t)= \frac{c_i c_j}{k_i + k_j} \exp\left( (k_i+k_j)x - (k_i^3 + k_j^3)t\right),\quad 1\leq i,j\leq 2.
$$
For any $t>0$, the total mass is equal to
$$
\int_\bR \rho(t,x) \dx = 4 (k_1+k_2).
$$
To illustrate the performance of our approach, we set
$$
T=2.5.10^{-3},\;
c_1 = 2,\;
c_2=3/2,\;
k_1 = 30 - k_2.
$$
The parameter domain is
$$
\rZ = \{ (t, k_2)\in [0, 2.5.10^{-3}] \times  [16, 22] \}
$$
It follows that the total mass is equal to 120 for all the parametric solutions.

\corr{Similarly as before, Figure \ref{fig:decay-svd-KdV} shows the decay of the error of the PCA and tPCA methods in the average sense and for the functions used in the training phase.} Figure \ref{fig:error-KdV} gives the errors in average and worst case sense for a set of 500 test functions. The observed behavior resembles the one of our previous examples. However, note that this time the average errors are roughly of order 10 while they were of order $10^{-2}$ or $10^{-3}$ in the previous examples. A visual inspection of the reconstructed functions reveals that PCA is in this case not producing ``worse-looking'' reconstructions than the nonlinear methods (see Figure \ref{fig:plotfun-KdV} of Appendix \ref{appendix:plotfun}). Two points can explain these observations: first, note that the exact solution does not present as sharp edges as in the other examples so this is playing in favor of PCA since it tends to produce oscillatory reconstructions. Second, in the present KdV evolution, we have two peakons of different masses propagating at different speeds. The fast peakon may overcome the slow one after a transition in which both peakons merge into a single one. Our strategy is based on a nonlinear mapping where translations can be treated simply but the mapping does not seem to be enough adapted to treat the case of fusion and separation of masses. This motivates to search for further nonlinear transformations to address this behavior. We outline some possibilities as well as further challenges in the next section.

\ploterrorall{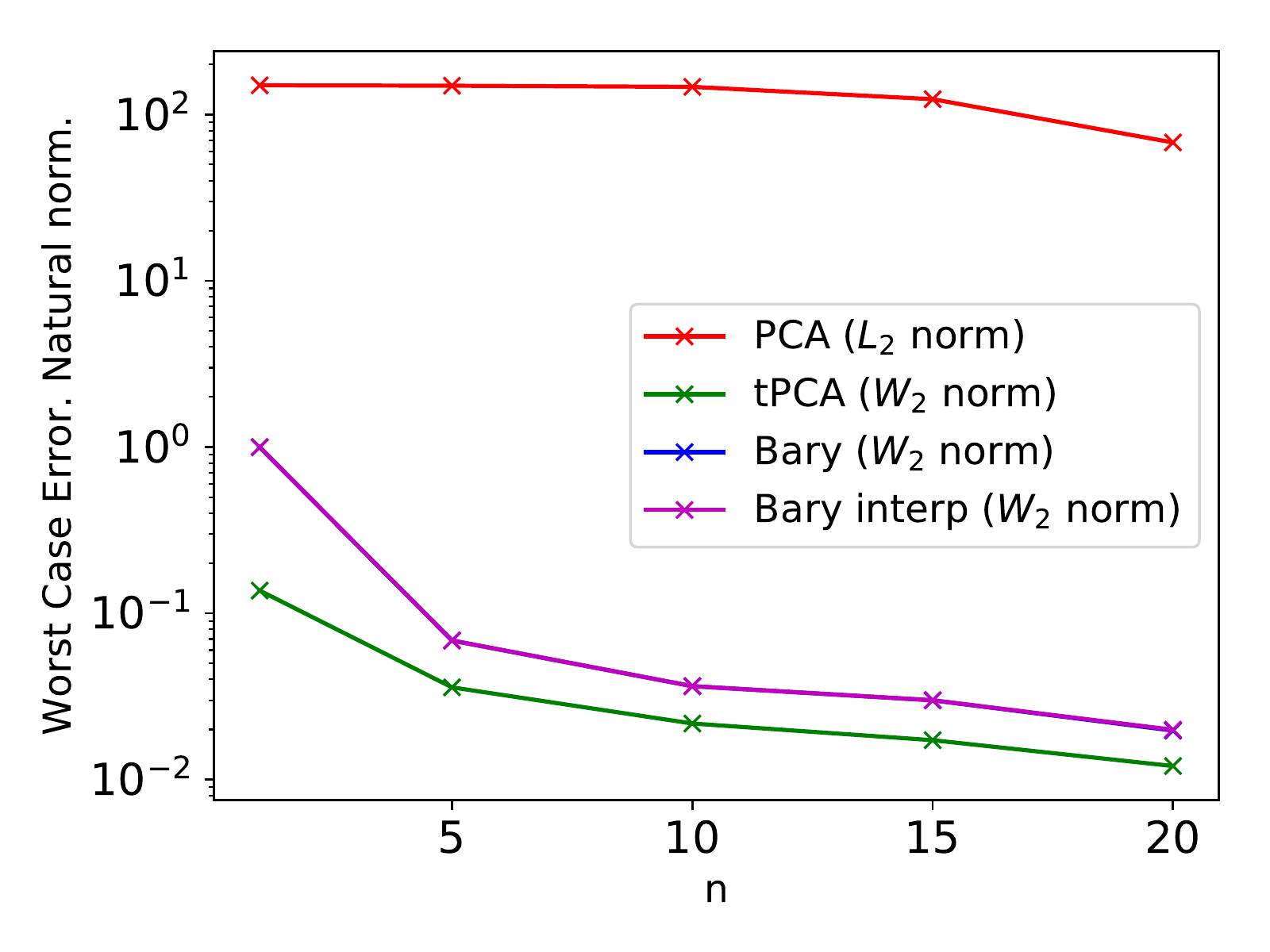}

\section{Conclusions and perspectives}
\label{sec:extensions}
In this work, we have shown how to exploit the geometry of some hyperbolic PDE in order to propose efficient reduced basis methods, one based on tangent PCA and one based on barycenters.
These two methods are developed using the Wasserstein metric which captures the advection effects. There are multiple possible extensions to 
this work, among which stand the following:

\begin{itemize}
 \item A first interesting theoretical question is the following: can the results of Section~\ref{sec:kolmo} on the decay of the Kolmogorov $n$-widths of the set $\mathcal{T}$ be extended to more general transport equations, and under which assumptions? 
 \item Can one obtain theoretical convergence rates of the greedy barycentric algorithm similar to the existing results on greedy algorithms for linear reduced bases~\cite{BCMN2018, MMT2016}?
 \item Can one reduce more complicated problems, i.e. defined on sets of dimension greater than one, or non-conservative problems, using one of the two presented algorithms? 
 This seems to be the case at least for the barycentric greedy algorithm. Indeed, one can consider for instance the Hellinger-Kantorovich distance~\cite{Mielke,CPSV2015}, which appears to yield interesting approximation properties 
 for problems where both transport and mass transfer phenomena occur. The computation of approximate barycenters of densities defined on spaces of dimension gerater than one can be done using entropic regularization 
 together with the Sinkhorn algorithm as 
 proposed in~\cite{CPSV2016} for instance.
 \item More generally, how to systematically select the best metric with respect to a given PDE and how to build non expensive surrogates of the exponential map? A promising direction to address these two issues seems to use machine learning algorithm to learn the metric for a given PDE \cite{2019arXiv190409524N}, and then to learn the associated exponential map \cite{2017arXiv171105766D,2019arXiv190600139S}.
\end{itemize}
We intend to address these issues in forthcoming works.

\bibliographystyle{plain}
\bibliography{references}

\corr{
\section*{Acknowledgements}}
\corr{The authors are grateful to the anonymous referees for their valuable comments which greatly improved the quality of the paper. This research was supported by the ANR JCJC grant COMODO, ANR JCJC grant ADAPT, and the EMERGENCES grant of the Paris City Council ``Models and Measures''.
}

\appendix

\corr{
\section{Alternative proof of (\ref{eq:kolmoL2trans})}
\label{appendix:proof}
}

\corr{
We give here an alternative proof to (\ref{eq:kolmoL2trans}) to the one given in~\cite{OR2016}.}

\corr{
\begin{proof}[Alternative proof of (\ref{eq:kolmoL2trans})]
Let $n\in \mathbb{N}^*$. The inequality $d_n(\cM, L^2(\Omega)) \geq \delta_n(\cM, L^2(\Omega)) $ is a direct consequence of the fact that for all finite-dimensional subspace $V_n\subset L^2(\Omega)$ of dimension $n$, 
 $$
\mathop{\sup}_{z\in \rZ} \| u(z) - P_{V_n}u(z) \|_{L^2(\Omega)} \geq \left( \int_{\rZ} \| u(z) - P_{V_n}u(z) \|^2_{L^2(\Omega)}\,dz\right)^{1/2}.
$$
Thus, we only have to prove that there exists a constant $c>0$ such that for all $n\in \mathbb{N}^*$, 
$$
\delta_n(\cM, L^2(\Omega)) \geq c n^{-1/2}. 
$$
 To obtain this result, we express $\delta_n(\cM, L^2(\Omega))$ as a function of the eigenvalues of the so-called correlation operator $K: L^2(\Omega) \to L^2(\Omega)$ which is defined as follows: 
$$
 \forall v \in L^2(\Omega), \quad (Kv)(x) = \int_{\Omega} \kappa(x,x') v(x')\,dx', 
 $$
where $\kappa(x,x')\coloneqq \int_{z\in \rZ} u(z)(x) u(z)(x')\,dy$. Since $u(z) = \charFun_{[z-1,z]}$, it holds that 
$$
\kappa(x,x') = \left\{
\begin{array}{ll}
 \max(0, 1- |x-x'|) & \mbox{ if } (x,x') \in [-1,0]\times[0,1] \cup [0,1] \times [-1,0],\\
 1 - \max(x,x') & \mbox{ if } (x,x') \in [0,1]\times[0,1] ,\\
 1 - \max(|x|,|x'|) & \mbox{ if } (x,x') \in [-1,0]\times[-1,0] ,\\
\end{array}\right.
$$
The function $\kappa$ is continuous and piecewise affine on $\Omega\times \Omega$, and $\kappa =0$ on $\partial( \Omega\times\Omega)$. We denote by $e_{k,k'}(x)\coloneqq \frac{1}{2}e^{i(kx+k'x')}$ for all $k,k'\in \mathbb{Z}$, so the the Fourier 
coefficient of $\kappa$ associated to the $(k,k')$ Fourier mode is defined by
$$
\alpha_{k,k'}\coloneqq\langle \kappa, e_{k,k'}\rangle.
$$
It can be easily checked that there exists a constant $C>0$ such that for all $k,k'\in \mathbb{Z}$,
\begin{equation}\label{eq:expre}
|\alpha_{k,k'}|\geq C \frac{1}{(|k| + |k'|)^2}.
\end{equation}
The correlation operator $K$ is a compact self-adjoint non-negative operator on $L^2(\Omega)$. Thus, there exists an orthonormal family $(f_k)_{k\in \mathbb{N}^*}$ and a non-increasing sequence of non-negative real numbers 
$(\sigma_k)_{k\in \mathbb{N}^*}$ going to $0$ as $k$ goes to $+\infty$ such that 
$$
K f_k =\sigma_k f_k, \quad \forall k\in \mathbb{N}^*.
$$
The scalars $\sigma_1 \geq \sigma_2 \geq \cdots \geq 0$ are the eigenvalues of the operator $K$, and it holds that 
$$
\delta_n(\cM, L^2(\Omega)) = \sqrt{\sum_{k\geq n+1} \sigma_k}.
$$
The $n^{th}$ eigenvalue $\sigma_n$ can be identified through the Max-Min formulas
$$
\sigma_n = \mathop{\max}_{
\begin{array}{c}
V_n \subset L^2(\Omega)\\
 \mbox{\rm dim}V_n= n\\
\end{array}}
\min_{\begin{array}{c}
        v_n\in V_n\\
        \|v_n\|_{L^2(\Omega)}=1\\
       \end{array}}
\langle v_n, Kv_n \rangle_{L^2(\Omega)},
$$
where $\langle \cdot, \cdot \rangle_{L^2(\Omega)}$ is the $L^2(\Omega)$ scalar product. We define $V_n\coloneqq \mbox{\rm Span}\left\{ \frac{1}{\sqrt{2}}e^{ik\cdot}, \quad k\in \mathbb{Z}, \; |k|\leq n\right\}$. Using \eqref{eq:expre}, it can easily be checked that there exists a constant $c'>0$ such that for all $n\in \mathbb{N}^*$, 
$$
\min_{\begin{array}{c}
       v_n\in V_n\\
        \|v_n\|_{L^2(\Omega)}=1\\
       \end{array}}
\langle v_n, Kv_n \rangle_{L^2(-1,1)}  =   \min_{\begin{array}{c}
        (\gamma_k)_{|k|\leq n}\in \mathbb{C}^{2n+1}\\
       \sum_{|k|\leq n}|\gamma_k|^2 = 1\\
       \end{array}} \sum_{|k|,|k'|\leq n} \alpha_{k,k'}\gamma_k \overline{\gamma_{k'}} \geq c' \frac{1}{n^2}.
$$
Thus, for all $n\in \mathbb{N}^*$, $ \sigma_n \geq c'n^{-2}$ and there exists a constant $c>0$ such that for all $n\in \mathbb{N}^*$, 
$$
\delta_n(\cM, L^2(\Omega)) = \sqrt{\sum_{k\geq n+1} \sigma_k} \geq c n^{-1/2}.
$$
 \end{proof}
 }

\section{Reconstruction plots}
\label{appendix:plotfun}

\plotfun{Burgers}{5}{10}
\plotfun{ViscousBurgers}{5}{10}
\plotfun{CH}{5}{10}
\plotfun{KdV}{5}{10}

\newpage

\section{\corr{Impact of the size of the training set $\cM_{\tr}$}}
\label{appendix:training-size}
\corr{
We investigate the impact of the size of the training set $\cM_\tr$ on the reconstruction errors of the Burgers' test case. For this, we randomly generate 10 realisations of training sets of cardinality
$$
\#\cM_\tr=10^2,\  5.10^2,\ 10^3,\ 3.10^3,\ 5.10^3,\ 7.10^3.
$$
For each realisation of $\cM_\tr$, we evaluate the error on a randomly generated test set $\cM_\test$ of 500 samples. Then we average the mean and wort case error over the 10 realizations of $\cM_\tr$ with the same cardinality.}

\corr{
Figure \ref{fig:size-av-natural} shows the mean error on in the natural norm of each algorithm ($L_2$ for PCA and $W_2$ for tPCA and the approach with the gBar algorithm). The general observed behavior is that, for a fixed dimension $n$, the approximation error tends to decrease as the size $\#\cM_\tr$ increases. Figure \ref{fig:size-av-natural} also shows that PCA and tPCA seem more robust to the resolution of $\cM_\tr$ than the approach with barycenters. The errors are particularly sensitive to the resolution of $\cM_\tr$ in the case when we interpolate the coefficients of the barycenter to obtain a fully online procedure. This behavior comes from our interpolation strategy: for each targeted snapshot $u(y)$ to reconstruct, we build a local interpolant based on its neighbors in $\cM_\tr$. As a result, the finer $\cM_\tr$ is, the closest the neighbors will be, and one can expect to obtain a more precise and stable procedure as is observed in the plots.
}

\corr{
Figures \ref{fig:size-av-Hminus1} to \ref{fig:size-wc-Hminus1} show the behavior of the mean and worst case errors for the natural norms of each algorithm and for the $H^{-1}$ norm. We may draw similar conclusions to the ones already discussed for Figure \ref{fig:size-av-natural}.
}

\newcommand{\isaverage}[1]{ \ifthenelse{\equal{#1}{av}}{Average}{Worst case}  }
\newcommand{\isnatural}[1]{ \ifthenelse{\equal{#1}{natural}}{Natural}{$H^{-1}$}  }

\newcommand{\sizeTraining}[2]{
\begin{figure}[htbp!]
\centering
\includegraphics[width=0.4\textwidth]{img/Burgers/size-training-average/#1-pca-#2-vs-sizeTraining}
\includegraphics[width=0.4\textwidth]{img/Burgers/size-training-average/#1-tpca-#2-vs-sizeTraining}
\includegraphics[width=0.4\textwidth]{img/Burgers/size-training-average/#1-bary-#2-vs-sizeTraining}
\includegraphics[width=0.4\textwidth]{img/Burgers/size-training-average/#1-bary-interp-#2-vs-sizeTraining}
\caption{Burgers' equation. Impact of the cardinality of $\cM_{\tr}$.~\protect\isaverage{#1} error over 10 realizations.~\protect\isnatural{#2} norm.}
\label{fig:size-#1-#2}
\end{figure}
}

\sizeTraining{av}{natural}
\sizeTraining{av}{Hminus1}
\sizeTraining{wc}{natural}
\sizeTraining{wc}{Hminus1}

\end{document}